\documentclass[11pt]{article}
\usepackage{setspace}
\usepackage[T1]{fontenc}
\usepackage{lmodern}
\usepackage[margin=1in]{geometry}
\usepackage{amsmath,amssymb,amsthm,mathtools}
\usepackage{booktabs,tabularx}
\usepackage{microtype}
\usepackage{cite}
\usepackage{algorithm}
\usepackage{algorithmic}

\usepackage[colorlinks=true,linkcolor=blue,citecolor=blue,urlcolor=blue]{hyperref}

\newlength{\ruleboxskip}
\numberwithin{equation}{section}
\newtheorem{theorem}{Theorem}[section]
\newtheorem{lemma}[theorem]{Lemma}
\newtheorem{proposition}[theorem]{Proposition}
\newtheorem{corollary}[theorem]{Corollary}
\theoremstyle{definition}
\newtheorem{assumption}[theorem]{Assumption}
\newtheorem{remark}[theorem]{Remark}
\newtheorem{example}[theorem]{Example}
\DeclareMathOperator{\Diag}{Diag}
\DeclareMathOperator{\sgn}{sgn}
\DeclareMathOperator*{\argmax}{arg\,max}
\newcommand{\R}{\mathbb R}
\newcommand{\one}{\mathbf 1}
\newcommand{\cZ}{\mathcal Z}
\newcommand{\cV}{\mathcal V}

\title{Sharp Linear Bounds and Exponential Separation for Single-Flip Methods in Absolute Value Equations}

\usepackage[marginal]{footmisc}
\usepackage{authblk}
\author[a]{Cairong Chen\thanks{Email address:  cairongchen@fjnu.edu.cn.}}
\author[b]{Yong Xia\thanks{Corresponding author. Email address: yxia@buaa.edu.cn.}}
\affil[a]{School of Mathematics and Statistics, Fujian Normal University, Fuzhou, 350117, People's Republic of China}
\affil[b]{School of Mathematical Sciences, Beihang University, Beijing 100191, People's Republic of China}

\date{\today}
\begin{document}
\maketitle
\begin{abstract}
Rohn's sign-accord algorithm is a classical single-coordinate sign-flip
method for absolute value equations $Ax-|x|=b$, selecting the least-index
mismatched coordinate at each step. We construct rational instances on
which this rule requires exactly $2^{n-1}-1$ flips from natural
initialization, even when both $\|A^{-1}\|_1$ and
$\|A^{-1}\|_\infty$ are arbitrarily small. We propose two alternative
selection rules, Gauss--Southwell and steepest single flipping,
which require only one flip on the same instances. A discrete potential
framework with positive polarity factors identifies solutions with global
maximizers. Under positive polarity with ratio $\kappa$ and
$r=\|A^{-1}\|_\infty<1$, both rules terminate from any initial signature
after at most $\lfloor s\kappa(1+r)/(1-r)\rfloor$ flips, where $s$ is the
number of nonzero solution coordinates. For fixed $\eta_1<1/2$ and $\eta_2<1$, both rules
have worst-case flip complexity $\mathcal{O}(n)$ on the class
$\|A^{-1}\|_1\le\eta_1$ and $\|A^{-1}\|_\infty\le\eta_2$. An $n-1$ lower bound from natural initialization shows that this linear
order is sharp. Under stronger, explicit sign-correction conditions, each rule makes exactly as many flips as there are initially incorrect signs among the nonzero solution components. Thus, each rule makes at most \(n\) flips, and this bound is optimal.
A five-dimensional example requiring six greedy flips shows that the
bound of $n$ does not extend to the entire common small-norm class. Building on Radons--Rump sign certificates, we further develop three classes of sign-flip algorithms using elimination guidance, cyclic
scanning, and randomized $m$-flipping. Under
$\rho(|A^{-1}|)<1/2$, we prove at most $n$ flips for elimination guidance,
$n^2$ inspections for cyclic scanning, and
$\mathcal{O}(n^2/m)$ expected batch updates for randomized $m$-flipping. Exact computations and retained paired experiments distinguish flip
counts from computational cost.
\end{abstract}
\noindent\textbf{Keywords:} absolute value equation; single flip;
Gauss--Southwell; steepest ascent; sign certification; linear iteration bound.

\section{Introduction}\label{sec:introduction}
We consider the absolute value equation (AVE) of the form
\begin{equation}\label{eq:ave}
 Ax-|x|=b,
\end{equation}
where $A\in\mathbb{R}^{n\times n}$ and $b\in\mathbb{R}^n$ are given,
and the absolute value is taken componentwise. AVEs are closely related
to linear complementarity problems, piecewise linear systems, and
bimatrix games \cite{mang2007a,mame2006,brca2008,huhu2010,prok2009}.
Although solving a general AVE is NP-hard \cite{mang2007a},
verifiable matrix conditions identify classes with a unique solution
and permit more specific algorithmic guarantees. For nonsingular $A$,
the normalized equation is
\begin{equation*}
 x-A^{-1}|x|=d,\qquad d=A^{-1}b,
\end{equation*}
which is in the form studied by \cite{rado2016,raru2022}. Over the past two decades, numerous algorithms have been developed for solving AVE~\eqref{eq:ave} which include optimization-based methods~\cite{mang2007,mang2007a,zahl2021}
and Newton-type methods~\cite{mang2009,caqz2011,zhwe2009,zahl2023,bcfp2016,gbrs2015,guo2025}. For further recent developments on AVEs and their extensions, we refer the reader to \cite{hmhk2026} and the references therein.

This paper focuses on methods that change one coordinate of a signature
$z\in\{-1,1\}^n$ at a time. For a fixed signature, the candidate
$x(z)=(A-\operatorname{Diag}(z))^{-1}b$ is obtained from a linear system.
It solves the AVE exactly when $z_i x_i(z)\ge0$ for every $i$.
A single-flip method selects a mismatched coordinate, for which
$z_i x_i(z)<0$, and reverses its signature. Such a change admits a
rank-one inverse update and costs $\mathcal{O}(n^2)$ dense arithmetic
operations after initialization. The central issue is how coordinate
selection affects the number of these updates.

To the best of our knowledge, Rohn's sign-accord algorithm
\cite{rohn1989,rohn2009} is the only previously proposed AVE algorithm
specifically formulated as single-coordinate mismatch flipping.
It chooses the mismatched coordinate with the smallest index.
Interval regularity ensures finite termination, but this guarantee
does not provide a polynomial bound on the number of flips.
We show that this limitation persists under arbitrarily small
simultaneous bounds on $\|A^{-1}\|_1$ and $\|A^{-1}\|_\infty$.
Thus strong contraction alone does not prevent an exponential
trajectory for this rule.

We propose Gauss--Southwell and steepest single-flip rule and analyze
their trajectories. A discrete potential provides
both an exact characterization of solutions and quantitative progress
estimates. In particular, comparing current mismatch magnitudes with
incorrectly signed solution amplitudes leads to a linear bound for the
two greedy rules on a common class with fixed inverse-norm margins.
This analysis concerns the specified coordinate rules and it does not
assert that the underlying matrix class was previously lacking a
polynomial solution method.

A complementary analysis uses the sign-identification and elimination
mechanism of Radons \cite{rado2016} and Radons--Rump \cite{raru2022}.
Their certificates identify a candidate sign uniformly over all
remaining signatures, and this property is preserved by elimination.
Although the single-flip interpretation is not stated in those works,
the mechanism naturally yields an elimination-guided single-flip rule.
We make this implication explicit and use the same certificate
structure to analyze cyclic scanning and randomized multiple-sign
flipping. The latter includes a randomized single-flip rule.

The main contributions are as follows.

\begin{itemize}
  \item We construct rational instances of polynomial encoding length
 on which Rohn's least-index rule makes exactly $2^{n-1}-1$ flips from
 natural initialization, while both inverse norms can be made
 arbitrarily small. Both proposed greedy rules solve the same instances
 in one flip.

  \item We establish a discrete potential formulation and prove that,
 under sign regularity and uniform polarity positivity, its maximizers
 correspond exactly to AVE solutions. If the polarity ratio is $\kappa$
 and $r=\|A^{-1}\|_\infty<1$, both greedy rules make at most
 $\lfloor s\kappa(1+r)/(1-r)\rfloor$ flips, where $s$ counts the nonzero
 solution coordinates. This gives $\mathcal{O}(n)$ worst-case flip
 complexity for fixed $\eta_1<1/2$ and $\eta_2<1$ on
 $\|A^{-1}\|_1\le\eta_1$ and $\|A^{-1}\|_\infty\le\eta_2$.
 An $n-1$ lower bound from natural initialization shows that the linear
 dimension dependence is sharp.

  \item We give explicit conditions under which each greedy rule makes
 exactly one flip per initially incorrect nonzero solution sign.
 Diagonal-sensitive estimates extend the sufficient
 conditions. A five-dimensional example requiring six greedy flips
 shows that the bound of $n$ does not hold throughout the common class
 with both inverse norms below $1/2$.

 \item From Radons--Rump certificates we derive elimination-guided,
 cyclic scanning, and randomized $m$-flip rules. Under $\rho(|A^{-1}|)<1/2$, elimination-guided single-flip rule exists with at most $n$ flips and cyclic scanning and randomized $m$-flips has $\mathcal{O}(n^2)$ and expected $\mathcal{O}(n^2/m)$ batch updates, respectively.
\end{itemize}

The remainder of the paper is organised as follows. Section~\ref{sec:preliminaries} introduces the notation and counting
conventions. Section~\ref{sec:formulation} establishes the potential
identities, solution equivalence, and general sign-flip framework.
Section~\ref{sec:linear_bounds} presents the greedy rules, their linear
bounds, and the lower-bound constructions. Section~\ref{sec:rr} develops
the certificate-based rules.
Section~\ref{sec:discussion} compares the guarantees, limitations, and
arithmetic costs, and Section~\ref{sec:numerics} reports the experiments.
Section~\ref{sec:conclusion} concludes the paper.
Appendix~\ref{sec:upper} gives a broader polarity-only bound.

\section{Basic definitions and preliminaries}\label{sec:preliminaries}
We collect the notation, sign conventions, and counting conventions used
throughout.

\subsection{Notation and matrix norms}
We write $[n]=\{1,\ldots,n\}$ and use $\mathbb{R}$ and $\mathbb{Q}$
for the real and rational numbers. For a finite set $S$, $\operatorname{card}(S)$ denotes its cardinality, and $R\setminus S$ denotes its complement with respect to $R$. For index sets $S,T$, $X_{ST}$ is the submatrix
with rows in $S$ and columns in $T$, and $x_S$ is the corresponding
subvector. We also use $X_{*i}$, $X_{i*}$, $X_{Ji}$, and $X_{iJ}$ for
the indicated columns and rows. The symbols $I$, $\mathbf{1}$, and $e_i$
denote the identity matrix, all-ones vector, and $i$th standard basis
vector, with dimensions determined by context.

The transpose, inverse and determinant of a matrix $X$ are
denoted by $X^T$, $X^{-1}$ and $\det(X)$, respectively; $X^{-T}=(X^{-1})^T$. Absolute values and inequalities for
vectors and matrices are understood componentwise.
For a scalar $a$, write $a_+=\max\{a,0\}$, and apply this notation
componentwise to vectors. The floor $\lfloor a\rfloor$ is the greatest
integer not exceeding $a$, and the ceiling $\lceil a\rceil$ is the
least integer not smaller than $a$.
We use `$\arg\max$' for the set of all maximizing indices and any tie-breaking
restrictions are stated with the relevant rule.

For a vector $x\in\mathbb{R}^n$,
\[
 \|x\|_1=\sum_{i=1}^n|x_i|,\qquad
 \|x\|_\infty=\max_{i\in[n]}|x_i|.
\]
For $X\in\mathbb{R}^{m\times n}$, the induced matrix norms are
\[
 \|X\|_1=\max_{j\in[n]}\sum_{i=1}^m|X_{ij}|,\qquad
 \|X\|_\infty=\max_{i\in[m]}\sum_{j=1}^n|X_{ij}|.
\]
We denote the spectral radius of a square matrix $B$ by  $\rho(B)$.

\subsection{Signatures, compatibility, and interval regularity}
For a vector $v$, $\operatorname{Diag}(v)$ is the diagonal matrix with
diagonal $v$. The signature set is $\mathcal{Z}=\{-1,1\}^n$, and we
write $Z=\operatorname{Diag}(z)$ for $z\in\mathcal{Z}$.
We adopt the componentwise convention: $\operatorname{sgn}(t) = 1$ if $t\ge 0$ and $\operatorname{sgn}(t) = -1$ if $t < 0$.
This convention selects one signature at a zero component and compatibility
itself allows either sign there. More precisely, $z$ is compatible
with $v$ if $z_i v_i\ge0$ for every $i$, equivalently $Zv=|v|$.

For $S\subseteq[n]$, $z^S$ denotes the signature obtained by reversing
the entries indexed by $S$ and leaving the other entries unchanged. In particular, $z^{\{i\}}$ reverses coordinate $i$.
When the selected set is clear, we write $z^+$ for the updated signature.

For the interval notation used below, matrix inequalities are
componentwise:
\[
 [A-I,A+I]
 =\{B:A-I\le B\le A+I\}
 =\{A-\operatorname{Diag}(\xi):\xi\in[-1,1]^n\}.
\]
This interval is regular if every matrix in it is nonsingular.
The sign-regularity assumption in Section~\ref{sec:formulation}
requires nonsingularity only at the signatures
$\xi\in\{-1,1\}^n$ while interval regularity concerns the entire diagonal
perturbation family.

\subsection{Update counts and probability conventions}
A single flip reverses one signature coordinate. A batch update
reverses a selected nonempty subset of current mismatches
simultaneously and is followed by a new candidate evaluation.
If $S_k$ is selected at update $k$, then $K$ batch updates make
$\sum_{k=0}^{K-1}{\rm card}(S_k)$ individual sign changes.
Thus updates and flips coincide for a single-flip rule.
The initial candidate solve is excluded from these counts.
For cyclic scanning, an inspection tests one coordinate for a mismatch,
and a sweep visits every coordinate once.

We write $\mathbb{E}$ for expectation and $P$ for probability.
In the randomized rule, both refer to the internal subset choices for
fixed problem data and a fixed initial signature.
The history $\mathcal{F}_k$ contains all information available before
the subset for update $k$ is drawn.
Uniform sampling is without replacement within each batch and is
conditional on that history. Almost sure termination means termination
with probability one. For a nonnegative integer-valued random variable
$W$, allowing $+\infty$, the tail-sum identity is
\[
 \mathbb{E}[W]=\sum_{h=0}^{\infty}P(W>h).
\]
The asymptotic notation $\mathcal{O}(n)$ refers to dimension dependence and
constants may depend on structural parameters explicitly held fixed.

\subsection{A rank-one update identity}
The determinant lemma and Sherman--Morrison formula supply the
single-coordinate inverse update \cite{shmo1950,rohn2009}.

\begin{lemma}[Sherman--Morrison rank-one update]\label{lem:sm}
Let $B\in\mathbb{R}^{n\times n}$ be nonsingular, let
$u,v\in\mathbb{R}^n$, and set $\delta=1+v^TB^{-1}u$. Then
\begin{itemize}
 \item[(i)] $\det(B+uv^T)=\delta\det(B)$.
 \item[(ii)] If $\delta\ne0$, then
 \[
  (B+uv^T)^{-1}
  =B^{-1}-\frac{B^{-1}uv^TB^{-1}}{\delta}.
 \]
\end{itemize}
\end{lemma}

\section{Discrete potential optimization: Formulation and equivalence}\label{sec:formulation}
The algebraic bottleneck of AVE~\eqref{eq:ave} lies in the nonsmooth absolute value operator. Let $x^*$ be the solution to AVE~\eqref{eq:ave}. If $z^* = \operatorname{sgn}(x^*)$ is known a priori, the nonsmooth problem \eqref{eq:ave} collapses into a trivial linear system $(A - \operatorname{Diag}(z^*)) x  = b$. Thus, an effective way to solve AVE~\eqref{eq:ave} is to identify the correct sign vector $z^* \in \cZ = \{-1,1\}^n$.

\subsection{Motivation and the discrete potential landscape}
For any $z \in \cZ$, assuming that the matrix $A - \operatorname{Diag}(z)$ is nonsingular, we define the induced primal vector as $x(z) = \left(A - \operatorname{Diag}(z)\right)^{-1}b$. A sign vector $z\in \cZ$ is called compatible with $x(z)$ if
\[
    z_i x_i(z)\ge 0
    \qquad\text{for every }i\in[n].
\]
Compatibility is equivalent to $Zx(z)=|x(z)|$ and hence to
$x(z)$ solving AVE~\eqref{eq:ave}. In particular, either sign is
allowed at a zero component. Define the mismatch set
\begin{equation}
\label{eq:mismatch_set}
    \cV_z:=\{i\in[n]:z_i x_i(z)<0\}.
\end{equation}
Then $\cV_z=\varnothing$ if and only if $x(z)$ solves AVE~\eqref{eq:ave}. Zero candidate coordinates are never mismatched. Our objective is to navigate the signature set~$\cZ$ to find a sign vector $z$ which is compatible with its induced primal vector $x(z)$. Geometrically, this compatibility can be quantified by their inner product $z^T x(z)$. Maximizing this compatibility yields our proposed discrete potential optimization~(DPO) problem
\begin{equation}
\label{eq:csso}
 \max_{z \in \cZ = \{-1,1\}^n} \quad F(z) := z^T (A - Z)^{-1} b,
\end{equation}
where $Z := \operatorname{Diag}(z)$ throughout the rest of this paper. Under the following sign regularity condition, the DPO problem \eqref{eq:csso} is well-defined.

\begin{assumption}[Sign regularity]\label{assu:regular}
For every $z\in \cZ$, $A-Z$ is nonsingular with $Z = \operatorname{Diag}(z)$.
\end{assumption}

\begin{lemma}[A sufficient condition for well-definedness]\label{lemma:well_definedness}
Suppose that $A$ is nonsingular and
$\rho(|A^{-1}|)<1$.
Then Assumption~\ref{assu:regular} holds.
\end{lemma}

\begin{proof}
For every $z\in\mathcal Z$, we have
\[
    \rho(A^{-1}Z)
    \le \rho(|A^{-1}Z|)
    =\rho(|A^{-1}|)
    <1.
\]
Consequently, $I-A^{-1}Z$ is nonsingular, and so is
$A-Z=A(I-A^{-1}Z)$.
\end{proof}

\subsection{Dual vector and uniform polarity positivity}\label{subsec:discrete_landscape}
Just as $x(z) = (A - Z)^{-1}b$ represents the primal vector induced by a sign vector $z\in \cZ$, we can symmetrically define the induced dual vector $y(z) = (A - Z)^{-T} z$. This definition reveals the underlying algebraic symmetry of the discrete potential function, allowing $F(z)$ to be evaluated from either the primal or dual perspective, i.e.,
\begin{equation*}
    F(z) = z^T x(z) = b^T y(z).
\end{equation*}
Beyond this structural symmetry, the dual vector $y(z)$ is used to constitute the polarity vector
\begin{equation}
\label{eq:polarity}
    p(z):=\mathbf 1+y(z)
          =\mathbf 1+(A-Z)^{-T}z,
\end{equation}
which governs the potential changes. When $A$ is nonsingular, we have
\begin{equation}
\label{eq:resolvent}
    p(z)=(I-A^{-T}Z)^{-1}\mathbf 1.
\end{equation}
Indeed, $(A^T-Z)y(z)=z$ implies $(I-A^{-T}Z)y(z)=A^{-T}z$, and~\eqref{eq:resolvent} follows from
$Z\mathbf 1=z$.

To guarantee strictly monotonic energy ascent dynamics, it is imperative that the components of the polarity vector $p(z)$ exhibit a positivity conservation, ensuring they do not collapse the structure of the landscape. Specifically, we impose the following uniform polarity positivity condition:

\begin{assumption}[Uniform polarity positivity]\label{ass:polarity}
For the polarity vector defined as in \eqref{eq:polarity}, there are constants $0<\alpha\le\beta$ such that
\[
 \alpha\le p_i(z)\le\beta, \qquad \forall i\in[n],\ \forall z\in\cZ.
\]
\end{assumption}

Since the signature set $\cZ$ is finite, existence of these constants is equivalent
to positivity of every polarity coordinate, under Assumption~\ref{assu:regular}.
Under Assumption~\ref{ass:polarity}, the quantitative bounds below depend on $\kappa = \beta/\alpha$ and a fixed bound on it is
needed for the stated dimension-only polynomial orders.

\begin{lemma}[Inverse-norm certificate]\label{lem:norm}
If $A$ is nonsingular and $\gamma=\|A^{-1}\|_1<1/2$, then Assumption~\ref{assu:regular}
and Assumption~\ref{ass:polarity} hold with
\begin{equation*}
 \alpha=\frac{1-2\gamma}{1-\gamma},\qquad
 \beta=\frac1{1-\gamma},\qquad
 \kappa=\frac{\beta}{\alpha} = \frac1{1-2\gamma}.
\end{equation*}
\end{lemma}

\begin{proof}
Since $\rho(|A^{-1}|) \le \|A^{-1}\|_1<1/2$, Assumption~\ref{assu:regular} holds by using Lemma~\ref{lemma:well_definedness}. The Neumann series in~\eqref{eq:resolvent} gives
\[
 \|p(z)-\one\|_\infty
 \le\frac{\|A^{-T}Z\|_\infty}{1-\|A^{-T}Z\|_\infty}
 =\frac\gamma{1-\gamma},
\]
from which the claimed bounds follow.
\end{proof}

\subsection{Algebraic potential transitions, AVE equivalence and sign-flip framework}\label{sec:potential}
We start from a discrete coordinate updated framework where the search trajectory advances exclusively via sign-flip. Let \(z\in\cZ\) denote the current sign vector. For any nonempty subset \(S\subseteq[n]\), let \(z^S\) denote the vector obtained by flipping the components of \(z\) indexed by \(S\). The fundamental geometric goal of this trajectory is to reach the global maximizer of the DPO problem~\eqref{eq:csso}.

To explore the discrete landscape, we evaluate, for each nonempty subset \(S\subseteq[n]\), the potential gap \(\Delta F_S = F(z^S) - F(z)\) induced by flipping the components of \(z\) indexed by \(S\).

\begin{lemma}[Exact transition identity]\label{lem:transition}
Under Assumption~\ref{assu:regular}, for any $z,\tilde{z}\in\cZ$,
\begin{equation}\label{eq:transition}
 F(\tilde{z}) - F(z)=\sum_{i=1}^n  (\tilde{z}_i - z_i ) x_i(z)p_i(\tilde{z}).
\end{equation}
Consequently, under Assumption~\ref{ass:polarity}, for
$\varnothing\ne S\subseteq\cV_z$ $($see \eqref{eq:mismatch_set}$)$,
\begin{equation}\label{eq:gainbounds}
 2\alpha\sum_{i\in S}|x_i(z)|
 \le  \Delta F_S = \sum_{i\in S}2|x_i(z)|p_i(z^S)
 \le2\beta\sum_{i\in S}|x_i(z)|.
\end{equation}
\end{lemma}

\begin{proof}
Write $\tilde{Z}=\Diag(\tilde{z})$ and $x(\tilde{z}) = (A - \tilde{Z})^{-1}b$. Since $(A-\tilde{Z})(x(\tilde{z})-x(z))=(\tilde{Z}-Z)x(z)$,
we have
\begin{align*}
 F(\tilde{z})-F(z)&= (\tilde{z}-z)^T x(z)+\tilde{z}^T(x(\tilde{z}) - x(z))\\
 &=(\tilde{z}-z)^T x(z)+ \tilde{z}^T (A-\tilde{Z})^{-1}(\tilde{Z}-Z)x(z)\\
 &= \sum_i(\tilde{z}_i - z_i)x_i(z)(1+y_i(\tilde{z}))\\
 &=\sum_i(\tilde{z}_i-z_i)x_i(z)p_i(\tilde{z}),
\end{align*}
which proves \eqref{eq:transition}.

For a subset flip, let $\tilde{z} = z^S$. Then $\tilde{z}_i - z_i=-2z_i$ on $S$ and zero elsewhere. For \(S\subseteq\cV_z\), each index \(i\in S\) contributes \(2|x_i(z)|p_i(z^S)\) to the gain. Summing these contributions and applying Assumption~\ref{ass:polarity} yields~\eqref{eq:gainbounds}.
\end{proof}

Although the equality in~\eqref{eq:gainbounds} provides an explicit characterization of \(\Delta F_S\), it is primarily of theoretical value and cannot directly yield a practical subset-selection criterion, since the post-flip vector \(z^S\) is unavailable at the current iterate. When \(S\) is a singleton, however, \(\Delta F_S\) admits an alternative representation in terms of quantities available at the current iteration.

Specifically, for $S = \{k\} \subseteq [n]$, the updated diagonal matrix is $Z^+ = Z - 2z_k e_k e_k^T$, from which we have $A - Z^+ = A - Z + 2z_k e_k e_k^T$. Let $H = (A - Z)^{-1}$ and $H^+ = (A - Z^+)^{-1}$. Applying the Sherman-Morrison formula (see Lemma~\ref{lem:sm}~(ii)), the updated resolvent $H^+$ is  given by
\begin{equation}\label{Hplus}
    H^+ = \left((A - Z) + 2z_k e_k e_k^T\right)^{-1}  = H - \frac{2z_k}{1 + 2z_k H_{kk}} H e_k e_k^T H.
\end{equation}
Assuming the global validity of the Sherman--Morrison formulation, we now derive an alternative transition equation governing the potential change upon a single flip.

\begin{corollary}[Discrete potential transition via single-flip]
\label{cor:discrete_jump}
Under Assumption~\ref{assu:regular}, for $z\in \cZ$,
flipping the $k$th coordinate to $z_k^+ = -z_k$ results in a potential variation
$$
    \Delta F_k = F(z^+) - F(z) = \frac{-2z_k x_k(z) p_k(z))}{1 + 2z_k H_{kk}},
    \label{trans}
$$
where $H_{kk}$ is the $k$th diagonal element of the resolvent $H = (A-Z)^{-1}$.
\end{corollary}

\begin{proof}
Based on \eqref{Hplus} and $x = x(z) = H b$, the updated primal vector $x^+$ expands as
\begin{equation}\label{eq:xplus}
    x^+ = H^+ b = x - \frac{2z_k (e_k^T H b)}{1 + 2z_k H_{kk}} H e_k = x - \frac{2z_k x_k}{1 + 2z_k H_{kk}} H e_k.
\end{equation}
The updated discrete potential is $F(z^+) = (z^+)^T x^+$. Substituting $z^+ = z - 2z_k e_k$, we have
\begin{equation*}
\begin{aligned}
    (z^+)^T x^+ &= (z - 2z_k e_k)^T \left( x - \frac{2z_k x_k}{1 + 2z_k H_{kk}} H e_k \right) \\
    &= z^T x - \frac{2z_k x_k}{1 + 2z_k H_{kk}} z^T H e_k - 2z_k e_k^T x + \frac{4 z_k^2 x_k}{1 + 2z_k H_{kk}} e_k^T H e_k.
\end{aligned}
\end{equation*}
Subtracting $F(z) = z^T x$, and substituting $z^T H e_k = y_k$, $e_k^T x = x_k$, $z_k^2 = 1$, and $e_k^T H e_k = H_{kk}$, the potential variation reduces to
\begin{equation*}
    \Delta F_k = -2z_k x_k - \frac{2z_k x_k y_k}{1 + 2z_k H_{kk}} + \frac{4 x_k H_{kk}}{1 + 2z_k H_{kk}}.
\end{equation*}
Factoring this expression over the common denominator $1 + 2z_k H_{kk}$ yields
\begin{equation*}
    \Delta F_k = \frac{-2z_k x_k (1 + 2z_k H_{kk}) - 2z_k x_k y_k + 4 x_k H_{kk}}{1 + 2z_k H_{kk}} = \frac{-2z_k x_k (1 + y_k)}{1 + 2z_k H_{kk}} = \frac{-2z_k x_k p_k(z)}{1 + 2z_k H_{kk}}.
\end{equation*}
This concludes the algebraic formulation of the potential transition.
\end{proof}

With the discrete potential transition established, we are now fully equipped to formally conclude the global equivalence.

\begin{theorem}[Solution equivalence and exact gap]\label{thm:gap}
Under Assumptions~\ref{assu:regular} and~\ref{ass:polarity},  AVE~\eqref{eq:ave} has a
unique solution $x^*$ for every $b$. Its compatible signatures are precisely
the maximizers of the DPO problem~\eqref{eq:csso}, and $F^*=\max_{z\in \cZ}\ F(z)=\|x^*\|_1$. Moreover, for any $z\in \cZ$, let $G(z)=F^*-F(z)$ and
$D(z)=\{i \in [n]:z_ix_i^*<0\}$. Then $G(z)=0$ if and only if $\cV_z =\varnothing$ and
\begin{equation}\label{eq:exactgap}
 2\alpha M(z) \le G(z)=2\sum_{i\in D(z)}|x_i^*|p_i(z) \le 2\beta M(z),
\end{equation}
where the total mismatch magnitude
\begin{equation}\label{eq:mz}
 M(z)=\sum_{i\in\cV_z}|x_i(z)|.
\end{equation}
\end{theorem}

\begin{proof}
A global maximizer exists on the finite set $\cZ$. Assume that $z^*\in \cZ$ is a global maximizer of the DPO problem~\eqref{eq:csso}, but $x^* = x(z^*) = (A - Z^*)^{-1} b$ is not the AVE solution. Because $x^*$ is not a solution of AVE~\eqref{eq:ave}, the mismatch set $\cV_{z^*}$ defined as in \eqref{eq:mismatch_set} must be nonempty. By~\eqref{eq:gainbounds}, a single mismatched flip would increase $F$, which contradicts the premise that $z^*$ is a global maximizer. Thus, any global maximizer $z$ must induce the  AVE solution $x(z) = (A-Z)^{-1}b$. Conversely, if $x^*$ is any solution to AVE~\eqref{eq:ave} and $z^*$ is compatible with it,
then $x(z^*)=x^*$ and~\eqref{eq:transition} gives
\begin{equation}\label{eq:ff}
 F(z)-F(z^*)=\sum_{i=1}^n (z_i-z_i^*)x_i^*p_i(z)
           =-2\sum_{i\in D(z)}|x_i^*|p_i(z)\le0,\qquad \forall z\in \cZ.
\end{equation}
Hence $z^*$ is a maximizer of the DPO problem~\eqref{eq:csso} and $F^* = F(z^*) = \|x^*\|_1$.  The equality in~\eqref{eq:exactgap} follows from \eqref{eq:ff}.

For another solution $\widehat x$ with
compatible signature $\widehat z$, both potentials have the same maximum
value. That is, $G(\widehat z) = F(z^*) - F(\widehat z) = 0$, which implies that $D(\widehat z) = \emptyset$ and hence $\widehat z$ is compatible with $x^*$ as well. Therefore
$(A-\widehat Z)x^*=(A-\widehat Z)\widehat x=b$, and nonsingularity
implies $x^*=\widehat x$. This also implies that $G(z)=0$ if and only if $\cV_z =\varnothing$.

Choose a global maximizer $z^*$ and define $J:=\{i\in[n]:z_i\ne z_i^*\}$.
The transition identity~\eqref{eq:transition} gives
\[
    G(z)
    =
    -2\sum_{i\in J}z_i x_i(z)p_i(z^*).
\]
Since $p_i(z^*)>0$, all terms with
$i\notin\mathcal V_z$ are nonpositive.
Discarding these terms yields
\[
    G(z)
    \le2\sum_{i\in J\cap\mathcal V_z}
           |x_i(z)|p_i(z^*)
    \le2\beta M(z).
\]

For the lower bound, suppose first that
$\mathcal V_z\ne\varnothing$.
Based on \eqref{eq:gainbounds}, the full-flip increment satisfies
\[
    2\alpha M(z)
    \le F(z^{\cV_z})-F(z)
    \le F^*-F(z) = G(z).
\]
If $\mathcal V_z=\varnothing$, both $M(z)$ and $G(z)$
are zero. The inequalities in \eqref{eq:exactgap} hold.
\end{proof}

\begin{remark}
If $x_i^*=0$, either choice of $z_i$ is compatible with
$x^*$. Thus the AVE solution is unique, whereas the
maximizing sign vector need not be unique.
If $x^*$ has exactly $r$ zero components, the DPO problem~\eqref{eq:csso}
has exactly $2^r$ maximizing sign vectors.
\end{remark}

\subsection{Sign-flip framework and its global finite termination}\label{sec:setup}
Under Assumption~\ref{assu:regular} and Assumption~\ref{ass:polarity}, for any $z\in\cZ$ with $\mathcal{V}_z\neq\emptyset$, flipping the components of $z$ indexed by any nonempty subset $S\subseteq\mathcal{V}_z$ strictly increases the potential (Lemma~\ref{lem:transition}). This property enables the sign-flip framework described in Algorithm~\ref{alg:totalAlgorithm}, which terminates at an AVE solution in finitely many steps.

\begin{algorithm}[h]
\caption{Framework of the sign-flip algorithms}
\label{alg:totalAlgorithm}
\begin{algorithmic}[1]
\REQUIRE Matrix $A$, vector $b$.

\STATE Initialize $z \in \{-1,1\}^n$, compute $Z = {\rm Diag}(z)$.

\STATE Solve $(A - Z) x = b$ to obtain $x$ and
 construct the mismatch set $\mathcal{V}_z = \{ i\in [n]: z_i  x_i < 0 \}.$

\WHILE{$\mathcal{V}_z\neq\emptyset$}

    \STATE Select a nonempty subset $S\subseteq \mathcal{V}_z$ by some rule.

    \STATE Sign-flip: $z_S \leftarrow - z_S$.

    \STATE Update $Z = {\rm Diag}(z)$, solve $(A - Z)x =b$, and update $\mathcal{V}_z$.

\ENDWHILE
\RETURN AVE solution $x$.
\end{algorithmic}
\end{algorithm}

\begin{theorem}[Global finite termination]
\label{thm:finite_termination}
Under Assumptions~\ref{assu:regular} and~\ref{ass:polarity}, any algorithm within the sign-flip framework of Algorithm~\ref{alg:totalAlgorithm} terminates after finitely many iterations at the unique solution to AVE~\eqref{eq:ave}.
\end{theorem}

\begin{proof}
By Theorem~\ref{thm:gap}, any sign vector $z\in \cZ$ that is not a global maximizer has a nonempty mismatch set $\mathcal{V}_z$. For such $z$, Lemma~\ref{lem:transition} implies that flipping any nonempty subset $S\subseteq\mathcal{V}_z$ yields a strict potential increase. Hence, each accepted update strictly increases $F(z)$. Because the signature set $\cZ$ is finite, no sign vector can be visited twice, and the algorithm must terminate after finitely many iterations at a point $z^*$ with $\mathcal{V}_{z^*} = \emptyset$. By Theorem~\ref{thm:gap}, such $z^*$ is a global maximizer of the DPO problem~\eqref{eq:csso} and $x^* = x(z^*)$ is the unique solution to AVE~\eqref{eq:ave}.
\end{proof}

\begin{remark}
For Algorithm~\ref{alg:totalAlgorithm}, one execution of the loop is an update, regardless of the size of $S$. For a single-flip rule, it is exactly one coordinate flip. Initialization
is not counted as an update. Thus $K$ updates require at most $K + 1$
candidate-system solves if no inverse updates are used. The natural initialization, when
$A$ is nonsingular, is $z  =\sgn(A^{-1}b)$.
\end{remark}

\section{Single-flip methods and their linear bounds}\label{sec:linear_bounds}
Different selection rules for the active set $S$ within the framework of Algorithm~\ref{alg:totalAlgorithm} yield distinct sign-flip algorithms. In this paper, we focus our attention on the single-flip rules.

As a baseline, Rohn's sign-accord algorithm~\cite{rohn2009} adopts a single-flip rule that picks the first mismatched coordinate, i.e., the least-index rule $S=\{i_{LI}\}$ with $i_{LI}=\min\{i:~ i\in\mathcal{V}_z\}$ (see \eqref{eq:mismatch_set}). Although computationally trivial, this index-based rule ignores the potential landscape and the resolvent-induced structure, potentially leading to inefficient trajectories and exponential worst-case iterations, even with arbitrarily strong contraction (see Theorem~\ref{thm:lower} below). To address this drawback, we propose two improved single-flip rules based on the information carried by $ \Delta F_i$:

\par\addvspace{\ruleboxskip}
\noindent\makebox[\linewidth][c]{%
  \fbox{%
    \begin{minipage}{0.92\linewidth}
\textbf{Steepest single-flip rule:}
\begin{equation*}
S=\{i_{ST}\} ~{\rm with}~ i_{ST} \in \arg\max_{i \in \mathcal{V}_z} \Delta F_i = \arg\max_{i \in \mathcal{V}_z} \frac{2 |x_i| p_i(z)}{1 + 2z_i H_{ii}}.
\end{equation*}
    \end{minipage}%
  }%
}
\par\addvspace{\ruleboxskip}

\par\addvspace{\ruleboxskip}
\noindent\makebox[\linewidth][c]{%
  \fbox{%
    \begin{minipage}{0.92\linewidth}
\textbf{Gauss--Southwell single-flip rule:}
\[
S=\{i_{GS}\} ~{\rm with}~i_{GS} \in \arg\max_{i \in \mathcal{V}_z} -z_ix_i = \arg\max_{i \in \mathcal{V}_z} |x_i|.
\]
    \end{minipage}%
  }%
}
\par\addvspace{\ruleboxskip}

More generally, based on \eqref{eq:exactgap}, the proposed single-flip rules can be extended to the following $c(n)$-capture rule:

\par\addvspace{\ruleboxskip}
\noindent\makebox[\linewidth][c]{%
  \fbox{%
    \begin{minipage}{0.92\linewidth}
\textbf{$c(n)$-capture rule:} a rule is a $c(n)$-capture rule if its selected set $\emptyset \neq S \subseteq \cV_z$ satisfies
\begin{equation}\label{eq:capture}
 \sum_{i\in S}|x_i(z)|\ge c(n) M(z)
\end{equation}
at each nonterminal state, where $c(n) \in (0,1]$ may depend on $n$ (but not on the iteration index).
    \end{minipage}%
  }%
}
\par\addvspace{\ruleboxskip}

\begin{remark}
The $c(n)$-capture rule also permits multiple coordinates to be flipped in a single update, although such updates are beyond the scope of this paper. For instance, selecting the largest
$\min\{m,{\rm card}(\cV_z)\}$ mismatch magnitudes has $c(n) = m/n$, where $m\in [n]$ is a given constant; the full flip rule (i.e., $m=n$) has $c(n) = 1$ and the Gauss--Southwell single-flip rule  has $c(n) = 1/n$.
\end{remark}

\begin{remark}[Implementing a single flip method]
\label{remark:fast_updates}
Let $H=(A-Z)^{-1}$ and $k\in \cV_z$. The Sherman--Morrison rank-one update formula gives \eqref{Hplus}, \eqref{eq:xplus} and
$$
 p^+ = p - \frac{2z_kp_k}{1+2z_kH_{kk}}H^Te_k.
$$
Here $1+2z_kH_{kk} \ne 0$ by Assumption~\ref{assu:regular} and Lemma~\ref{lem:sm}~(i).
In particular $p_k^+=p_k/(1+2z_kH_{kk})$, and Assumption~\ref{ass:polarity} yields
\[
 \kappa^{-1}\le 1+2z_kH_{kk} =\frac{p_k}{p_k^+}\le\kappa,
 \qquad
 F(z^+)-F(z)=\frac{2|x_k|p_k}{1+2z_kH_{kk}}.
\]
Thus all steepest scores are computable from the current $x,p$ and the
diagonal of $H$, without a separate linear solve for each trial flip.
Maintaining $H$ costs $\mathcal{O}(n^2)$ arithmetic operations per actual flip,
after an $\mathcal{O}(n^3)$ initialization. These are arithmetic-operation statements whereas
finite-precision stability is a separate issue.
\end{remark}

Combining \eqref{eq:exactgap} with the potential gain bounds allows us
to quantify progress relative to the remaining gap. This yields a
uniform contraction estimate for capture-based updates, as well as
for the Gauss--Southwell single-flip and steepest single-flip rules.

\begin{proposition}[Uniform contraction]\label{prop:contraction}
Under Assumptions~\ref{assu:regular} and~\ref{ass:polarity}, every \(c(n)\)-capture update satisfies
\begin{equation}\label{eq:contractcapture}
 G(z^+)\le \left(1-\frac{c(n)}{\kappa}\right)G(z).
\end{equation}
Moreover, at every nonterminal state \(z\), both the Gauss--Southwell and steepest single-flip rules produce an update satisfying
\begin{equation}\label{eq:contractsingle}
 G(z^+)\le \left(1-\frac{1}{n\kappa}\right)G(z).
\end{equation}
\end{proposition}

\begin{proof}
For the $c(n)$-capture rule, equations~\eqref{eq:gainbounds}, \eqref{eq:exactgap} and \eqref{eq:capture} give
$F(z^+)-F(z)\ge2\alpha c(n) M(z)\ge(c(n)/\kappa)G(z)$, from which we can deduce \eqref{eq:contractcapture}.
For the Gauss--Southwell single-flip rule,
\[
    F(z^+)-F(z)
    \ge2\alpha\max_{i\in\mathcal V_z}|x_i(z)|
    \ge\frac{2\alpha}{{\rm card}(\mathcal V_z)}M(z)
    \ge\frac{G(z)}{\kappa{\rm card}(\mathcal V_z)}\ge \frac{G(z)}{n\kappa},
\]
which implies \eqref{eq:contractsingle}.
The steepest single-flip rule attains at least the potential increment of the Gauss--Southwell single-flip rule, so it has the same bound.
\end{proof}

\begin{remark}
The direct comparison for steepest single-flip rule avoids an extra factor of
$\kappa$ that would result from first converting steepest ascent into a
magnitude-capture estimate. Capture is a sufficient progress condition and
no necessity claim is made here.
\end{remark}

\subsection{Linear bounds from the largest incorrect amplitude}\label{sec:linear}
The preceding contraction estimates yield polynomial termination bounds
under uniform polarity positivity, as detailed in
Appendix~\ref{sec:upper}. To obtain a linear bound for the two greedy
single-flip rules (the steepest single-flip rule and the Gauss--Southwell single-flip rule), we additionally assume $\|A^{-1}\|_\infty<1$ and
compare the current mismatch magnitudes with the largest solution
amplitude whose sign is still incorrect. This comparison guarantees
sufficient potential gain at each update. Combined with a concave counting function based on the solution amplitudes, this yields an explicit upper bound on the number of flips
that is linear in the number of nonzero solution coordinates, without
requiring any positive lower bound on their magnitudes.

\begin{lemma}[A correcting coordinate of large magnitude]\label{lem:largest}
Suppose that $r=\|A^{-1}\|_\infty<1$. Then AVE~\eqref{eq:ave} has a unique solution $x^*$ for every $b\in \mathbb{R}^n$ and, at every nonterminal signature $z\in \cZ$, there is
$i\in D(z)\cap\cV_z$ such that
\begin{equation}\label{eq:largest}
 m_i(z)\ge\tau a(z) \qquad \text{with}~~\tau=\frac{1-r}{1+r},
\end{equation}
where
\begin{equation}\label{eq:ma}
m_i(z)=(-z_ix_i(z))_+,\qquad u = |x^*|, \qquad D(z)=\{i\in [n]: z_ix_i^*<0\}, \qquad a(z)=\max_{i\in D(z)}u_i.
\end{equation}
\end{lemma}

\begin{proof}
Since $r=\|A^{-1}\|_\infty<1$, AVE~\eqref{eq:ave} has a unique solution $x^*$ for any $b\in \mathbb{R}^n$. For $z\in \cZ$ with $D(z) \neq \emptyset$, put $w=Zx(z)$ and let $u = |x^*|$ and $u_D$ agree with $u$
on $D(z)$ and vanish elsewhere. Choosing any signature $Z^*$ compatible with
$x^*$ (i.e., $Z^* x^* = |x^*|$), the equations $(Z- A^{-1})w=(Z^*- A^{-1})u = A^{-1}b$ imply
\[
 (I-ZA^{-1})(w-u)=(ZZ^* - I)u = -2u_D.
\]
Let $v=(I-ZA^{-1})^{-1}u_D$, then we have
\begin{equation}\label{eq:videntity}
v=u_D+ZA^{-1}v,\qquad w = u-2v.
\end{equation}
Take $i$ with $|v_i|=V:=\|v\|_\infty>0$. The equality $v_i=-V$ is impossible,
because $u_D \ge 0$ would give $-V \ge -|(ZA^{-1}v)_i| \ge -rV$, which contradicts with $r<1$. In addition, $i\notin D(z)$ would imply
$V=|(ZA^{-1}v)_i|\le rV$, which also contradicts with $r<1$. Thus $i\in D(z)$ and $v_i=V$.
Now $u_i= (u_{D})_i = V-(ZA^{-1}v)_i\le(1+r)V < 2V$, so $w_i = u_i - 2V < 0$ and hence $i\in \cV_z$. Moreover,
\begin{equation}\label{ie:mz}
 m_i(z) = (-w_i)_+ = 2V-u_i\ge(1-r)V>0.
\end{equation}
Finally, $a(z)=\|u_D\|_\infty=\|v-ZA^{-1}v\|_\infty\le(1+r)V$.
Combining this bound with~\eqref{ie:mz} yields~\eqref{eq:largest}.
\end{proof}

\begin{lemma}[Concave amplitude counting]\label{lem:amplitudecount}
Assume that Assumptions~\ref{assu:regular} and~\ref{ass:polarity} hold. Let $u$ and $a(z)$ be defined as in \eqref{eq:ma} with $x^*$ being the unique solution to AVE~\eqref{eq:ave}. Suppose that a trajectory started at $z^{(0)}\in \cZ$ satisfies, for some
$\lambda>0$,
\begin{equation}\label{eq:amplitudeprogress}
 G(z^{(k)}) - G(z^{(k+1)}) \ge 2\alpha\lambda a(z^{(k)})
\end{equation}
at every nonterminal state $z^{(k)}\in \cZ$. If $s={\rm card}(\{i\in [n]:x_i^*\ne0\})$, then the number $K$ of updates satisfies
\begin{equation}\label{eq:amplitudecount}
K \le\left\lfloor\frac\kappa\lambda
 \Phi\!\left(\frac{G(z^{(0)})}{2\beta}\right)\right\rfloor
 \le\left\lfloor\frac{\kappa s}{\lambda}\right\rfloor,
 \qquad \kappa=\beta/\alpha.
\end{equation}
Here, for sorted amplitudes $0< u_{i_1}\le\cdots\le u_{i_s}$ and
$B_j=\sum_{k=1}^j u_{i_k}$, $B_0=0$,
\begin{equation}\label{eq:countpotential}
 \Phi(h)=j-1+\frac{h-B_{j-1}}{u_{i_j}} \qquad \text{with}~~B_{j-1}\le h\le B_j~~\text{and}~~j = 1,2,\ldots, s.
\end{equation}
Particularly, if $s = 0$, then $K =0$.
\end{lemma}

\begin{proof}
For $s>0$, $\Phi$ is continuous, increasing and concave, with $\Phi(0)=0$
and $\Phi(B_s)=s$. For any $z^{(k)}\in \cZ$, let $h(z^{(k)}) = G(z^{(k)})/(2\beta)$. Then the exact gap identity in~\eqref{eq:exactgap} implies $h(z^{(k)}) \in [0,B_s]$. If
$h(z^{(k)}) \in(B_{j-1},B_j]$, then $a(z^{(k)})\ge u_{i_j}$. Otherwise, $a(z^{(k)})< u_{i_j}$ implies that $i_k \notin D(z^{(k)})$ for $k\in j, j+1,\ldots,s$. In other words, all incorrectly signed coordinates are among the first $j-1$ sorted entries. This implies that $G(z^{(k)})\le2\beta B_{j-1}$, which contradicts with $h(z^{(k)}) > B_{j-1}$. Hence, $a(z^{(k)})\ge u_{i_j}$ and the condition \eqref{eq:amplitudeprogress} imply that
\begin{equation*}
h(z^{(k)}) - h(z^{(k+1)} )\ge(\lambda/\kappa)u_{i_j} > 0.
\end{equation*}

On $[ h(z^{(k+1)}),h(z^{(k)})]$, every linear segment of $\Phi$ has slope at least $1/u_{i_j}$.
It follows that $\Phi(h(z^{(k)}))-\Phi(h(z^{(k+1)}))\ge (h(z^{(k)}) - h(z^{(k+1)}))/u_{i_j} \ge \lambda/\kappa$.
Summation until termination proves~\eqref{eq:amplitudecount}. Alternatively,
this positive decrease of a bounded nonnegative quantity itself proves
finite termination. If $s=0$, then $b=0$ and all candidates $x(z)$ are zero.
\end{proof}

\begin{theorem}[Linear greedy single-flip bounds]\label{thm:linear}
Assume that Assumptions~\ref{assu:regular} and~\ref{ass:polarity} hold and
$r=\|A^{-1}\|_\infty < 1$. Let $s={\rm card}(\{i\in [n]:x_i^*\ne0\})$  with $x^*$ being the unique solution to AVE~\eqref{eq:ave}. Starting from any initial signature, the Gauss--Southwell and steepest single-flip methods both terminate after at most
\begin{equation}\label{eq:linearbound}
K \le \left\lfloor s\kappa\frac{1+r}{1-r}\right\rfloor
\end{equation}
flips. The bound holds for any choice among tied maximizers.
\end{theorem}

\begin{proof}
It follows from~\eqref{eq:gainbounds} and Lemma~\ref{lem:largest} that the Gauss--Southwell increment is at least $G(z^{(k)}) - G(z^{(k+1)}) = F(z^{(k+1)}) - F(z^{(k)}) \ge 2\alpha\max_{i\in \cV_{z^{(k)}}} |x_i(z^{(k)})| \ge 2\alpha\max_{i\in \cV_{z^{(k)}}\bigcap D(z^{(k)})} |x_i(z^{(k)})| = 2\alpha\max_{i\in \cV_{z^{(k)}}\bigcap D(z^{(k)})} m_i(z^{(k)})\ge 2\alpha\tau a(z^{(k)})$. Steepest attains at least the increment of the Gauss--Southwell choice at the same state. Then \eqref{eq:linearbound} follows from \eqref{eq:amplitudecount} with $\lambda=\tau$ (see \eqref{eq:largest}).
\end{proof}

\begin{corollary}[A small-norm class]\label{cor:linearcommon}
If $\|A^{-1}\|_1\le\eta_1<\frac12$ and $\|A^{-1}\|_\infty\le\eta_2<1$, then, from any initial signature, the Gauss--Southwell and steepest single-flip methods both terminate after at most
\begin{equation}\label{eq:gsst}
K
\le \left\lfloor
\frac{1+\eta_2}
{(1-\eta_2)(1-2\eta_1)}\,n
\right\rfloor =\mathcal{O}_{\eta_1,\eta_2}(n)
\end{equation}
flips.
\end{corollary}

\begin{proof}
The result follows directly from Lemma~\ref{lem:norm} and Theorem~\ref{thm:linear}.
\end{proof}

\begin{remark}
The assertion of Lemma~\ref{lem:largest} does not require Assumption~\ref{ass:polarity}. In \eqref{eq:gsst}, the order in $n$ is sharp, including for natural initialization, by
Corollary~\ref{cor:theta} below.
\end{remark}

\subsection{At most one correction per initially incorrect sign}\label{sec:unit}
The preceding bounds control the total number of flips through the
increase of a potential. Under stronger conditions, every flip selected by the greedy single-flip rules corrects an incorrect solution sign, and each corrected sign remains correct throughout the rest of the trajectory. In this regime, the flip count equals the number
of initially incorrect signs, giving a shortest path to a compatible
signature.

\begin{theorem}[Shortest sign-correction paths]\label{thm:unit}
Let $r=\|A^{-1}\|_\infty<1$ and $z^{(0)}$ be the initial signature. Assume that Assumption~\ref{assu:regular} holds.
\begin{enumerate}
\item [(i)] If \(r<1/3\), the Gauss--Southwell single-flip rule terminates after exactly \({\rm card}(D(z^{(0)}))\le s\le n\) flips. Every selected coordinate corresponds to a nonzero solution component whose sign is initially incorrect, and each such coordinate is flipped exactly once.

\item [(ii)] Under Assumption~\ref{ass:polarity}, the steepest single-flip rule has the same property if $2\kappa r<1-r$.
\end{enumerate}
\end{theorem}

\begin{proof}
Let $u,v$ and $V$ be defined as in the proof of Lemma~\ref{lem:largest}. Some $j\in D(z)= \{i\in [n]: z_i x_i^*< 0\}$ has
$m_j(z)\ge(1-r)V$. For every $i\notin D(z)$,
$v_i=(ZA^{-1}v)_i$, so $m_i(z)=(2v_i-u_i)_+\le 2rV$.
If $r<1/3$, then $(1-r)V > 2rV$ and the Gauss--Southwell maximizer must lie in $D(z)$.
For steepest single-flip rule, \eqref{eq:gainbounds} implies that a correcting coordinate has increment at least
$2\alpha(1-r)V$ whereas any noncorrecting mismatch has increment at most
$4\beta rV$. The strict inequality in the assumption excludes the latter.
Each step therefore removes exactly one member of $D(z)$ and never adds
one. Once $D(z)$ is empty, $z$ is compatible with $x^*$ and its candidate
is $x^*$. If it is nonempty, the correcting coordinate ensures an admissible
step. This proves the exact count, including zero solution components.
\end{proof}

\begin{corollary}\label{cor:unitcommon}
Assume that \(\max\{\|A^{-1}\|_1,\|A^{-1}\|_\infty\}\le \eta<1/2\) and $z^{(0)}$ is the initial signature. If \(\eta<1/3\), the Gauss--Southwell single-flip rule terminates after exactly \({\rm card}(D(z^{(0)}))\) flips. The same exact flip count holds for the steepest single-flip rule when
\[
\eta<\frac{5-\sqrt{17}}{4}\approx 0.2192236.
\]
The unit leading coefficient in these flip bounds is optimal.
\end{corollary}

\begin{proof}
The claim of the Gauss--Southwell single-flip rule follows directly from Theorem~\ref{thm:unit}.
For the steepest rule, the result is obtained by substituting $\kappa\le(1-2\eta)^{-1}$ into
$2\kappa\eta<1-\eta$. Optimality follows from Corollary~\ref{cor:theta}.
\end{proof}

\subsection{Separating the diagonal from coupling}
The full-norm conditions of Theorem~\ref{thm:unit} combine diagonal
contributions with off-diagonal coupling. Separating these contributions
gives a more detailed comparison between correcting and noncorrecting
mismatches. This refinement yields exact sign-correction counts for
nearly diagonal inverse matrices whose full norms may lie outside the
preceding sufficient regimes.

\begin{theorem}[Diagonal-sensitive correction certificates]\label{thm:diagonal}
Let $C = A^{-1}$ and $z^{(0)}$ be the initial signature.. For $|c_{ii}|<1$, define
\[
 d_i=|c_{ii}|,\qquad h_i=\sum_{j\ne i}|c_{ij}|,\qquad
 \ell=\min_i(1-d_i-h_i),\qquad q=\max_i\frac{h_i}{1-d_i}.
\]
If $2q<\ell$, the Gauss--Southwell single-flip rule makes exactly ${\rm card}(D(z^{(0)}))$ flips.
Under Assumption~\ref{ass:polarity}, $2\kappa q<\ell$ gives the same result
for the steepest single-flip rule.
\end{theorem}

\begin{proof}
Either condition implies $\ell>0$ and hence $\|A^{-1}\|_\infty<1$.
For $z\in \cZ$, at a positive maximizer $v_j=V$ from~\eqref{eq:videntity},
\[
 m_j(z) =(1+z_jc_{jj})V+z_j\sum_{k\ne j}c_{jk}v_k
 \ge(1-d_j-h_j)V\ge\ell V.
\]
For $i\notin D(z)$,
$(1-z_ic_{ii})v_i=z_i\sum_{j\ne i}c_{ij}v_j$, so $m_i\le2qV$. Then
the magnitude and increment comparisons in Theorem~\ref{thm:unit} apply.
\end{proof}

\begin{corollary}[A near-diagonal class]\label{cor:neardiagonal}
Let $C= A^{-1} = D+E$, where $D$ is diagonal, $\Diag(E)=0$,
$\|D\|_\infty\le t<1$, and $\|E\|_\infty\le\varepsilon$. Let $z^{(0)}$ be the initial signature. Then the Gauss--Southwell single-flip rule  has the exact count ${\rm card}(D(z^{(0)}))$ whenever
\[
 \varepsilon<\frac{(1-t)^2}{3-t}.
\]
If also $\|E\|_1\le\varepsilon<(1-t)/2$, polarity is positive with
\[
 \beta=\frac1{1-t-\varepsilon},\qquad
 \alpha=\frac{1-t-2\varepsilon}{(1-t-\varepsilon)(1+t)},\qquad
 \kappa=\frac{1+t}{1-t-2\varepsilon}.
\]
Then the steepest rule has the same exact count if, in addition,
$2\kappa\varepsilon/(1-t)<1-t-\varepsilon$.
\end{corollary}

\begin{proof}
Since $\ell\ge1-t-\varepsilon$ and $q\le\varepsilon/(1-t)$, $\varepsilon< {(1-t)^2}/{(3-t)}$ implies $2q < \ell$, which proves the claim for the Gauss--Southwell single-flip rule.

For polarity, the Neumann bound gives $\|p\|_\infty\le\beta$ and
\[
 p_i=\frac{1+\sum_{j\ne i}c_{ji}z_jp_j}{1-c_{ii}z_i}
 \ge\frac{1-\varepsilon\beta}{1+t}=\alpha>0.
\]
The result for the steepest single-flip rule then follows from Theorem~\ref{thm:diagonal}.
\end{proof}

\begin{remark}
This refinement yields exact sign-correction counts for
nearly diagonal inverse matrices whose full norms may lie outside the
 sufficient regimes in Corollary~\ref{cor:unitcommon}. For example, $t=0.8$, $\varepsilon=0.001$ guarantees the exact count for
both rules, although the full matrix norms may exceed $1/2$. For the numerical parameters,  $\varepsilon=0.001 < {(1-t)^2}/{(3-t)} \approx 0.018$. In addition, $\ell \ge 0.199$, $q \le 0.005$ and
$\kappa \le 100/11$, so $2 \kappa q\le1/11<0.199$. The conditions of Corollary~\ref{cor:neardiagonal} are satisfied.
\end{remark}

\subsection{Sharpness and the limits of a unit coefficient}\label{sec:tight}
The results above establish linear upper bounds and give sufficient
conditions for the exact count ${\rm card}(D(z^{(0)}))$. We now examine how
much these conclusions can be strengthened. A construction with natural
initialization shows that the linear dependence on dimension is sharp.
A further example, together with block-diagonal replication, shows that
a universal unit leading coefficient fails for arbitrary initialization
even when both $\|A^{-1}\|_1$ and $\|A^{-1}\|_\infty$ are below $1/2$.

\begin{theorem}[Natural initialization needs $n-1$ flips]\label{thm:tightnatural}
For every fixed $\eta>0$ and every $n\ge2$, there are nonsingular rational matrix
$A^{-1}\in \mathbb{Q}^{n\times n}$ and rational vector $d\in \mathbb{Q}^n$ with $\max \{\|A^{-1}\|_1, \|A^{-1}\|_\infty\} < \eta$ such that every
single-mismatch-flip rule, starting from $z^{(0)}= \sgn(d)$, makes exactly $n-1$
flips. The data have polynomial binary encoding length.
\end{theorem}

\begin{proof}
Choose a fixed rational $0<t<\min\{\eta/2,1/2\}$ and set $c=t/(n-1)$,
\[
 A^{-1} =tI+c\sum_{i=1}^{n-1}e_ie_n^T,\qquad
 d_i=-\frac{c}{2(1-t)}\ (i<n),\qquad d_n=1.
\]
Then $\|A^{-1}\|_1=2t < \eta$, $\|A^{-1}\|_\infty=t+c\le 2t < \eta$, and $A^{-1}$ is triangular
with positive diagonal. Natural initialization is $z^{(0)} = (-1,\ldots,-1,1)^T$.
For every signature $z\in \cZ$ with $z_n=1$,
\[
 x_n(z)=\frac1{1-t}>0,\qquad
 x_i(z)=\frac{c}{2(1-t)(1-tz_i)}>0\quad(i<n).
\]
Hence, starting from $z^{(0)}$, the last coordinate is never selected. Each other coordinate is mismatched exactly while its sign is negative and its flip does not change any other
candidate coordinate. Each of these $n-1$ coordinates is therefore flipped
once. Rational encoding is polynomial, also for $A$ and $b=Ad$.
\end{proof}

\begin{corollary}[Sharp dimension dependence]\label{cor:theta}
For every fixed $0<\eta<1/2$, the worst-case flip counts of the Gauss--Southwell and the steepest single-flip rules on $\max\{\|A^{-1}\|_1, \|A^{-1}\|_\infty\}\le\eta$ are $\mathcal{O}(n)$,
even with natural initialization. Moreover, if $\eta < {(5-\sqrt{17})}/{4}$,  no bound $cn+o(n)$ with $c<1$ is possible.
For arbitrary initialization the exact bound $n$ is attained by
$A^{-1}=tI$, $d>0$, and $z^0=-\one$, with any fixed sufficiently small $t>0$.
\end{corollary}

\begin{proof}
The first two assertions follow from Corollary~\ref{cor:linearcommon}, Corollary~\ref{cor:unitcommon} and Theorem~\ref{thm:tightnatural}.
For the diagonal example, every candidate is positive, independently of
its signature, and one negative sign is corrected per flip.
\end{proof}

\begin{proposition}\label{prop:six}
The Gauss--Southwell and the steepest single-flip rules can require more than $n$ flips when $\max \{\|A^{-1}\|_1, \|A^{-1}\|_\infty\} < 1/2$.
\end{proposition}

\begin{proof}
We construct an example in dimension five for which both greedy single-flip rules require six flips. Specifically, let
\begin{equation}\label{eq:5e}
 A^{-1} =\frac1{400}\begin{pmatrix}
 0&-49&-49&-49&-49\\
 -49&147&0&0&0\\
 -49&0&147&0&0\\
 -49&0&0&147&0\\
 -49&0&0&0&147
 \end{pmatrix},~~
 x^*=(-10^{-4},1,1,1,1)^T,~~ d=x^*-A^{-1}|x^*|.
\end{equation}
The matrix satisfies $\|A^{-1}\|_1=\|A^{-1}\|_\infty=49/100$. From $z^{(0)} = -\one$, both rules flip coordinates in the order
$1,2,3,4,5,1$, up to permutation of $2,3,4,5$.

Indeed, at initialization, exact candidate magnitudes are
\[
 x_1=\frac{391947701}{522990000}>\frac7{10},\qquad
 x_j=\frac{27701}{52299}<\frac35\quad(j>1).
\]
The exact increments satisfy
\[
 \Delta F_1  = \frac{291217141843}{149316259950}>\frac{39}{20},\qquad
 \Delta F_j  = \frac{1360689740600}{717115886253}<\frac{19}{10}\quad(j>1).
\]
Thus both first select coordinate 1, which initially has the correct true
sign. To verify the remaining path put $a=49/400$, $h=147/400$,
$e=10^{-4}$, $d_1=4a-e$, and $d_s=1-h+ae$.
After that first flip, suppose $k$ of the four other signs are positive.
Writing $S_k=k/(1-h)-(4-k)/(1+h)$ gives
\[
 x_1=\frac{d_1-aS_kd_s}{1-a^2S_k},\qquad
 x_j=\frac{d_s-ax_1}{1-hz_j}\quad(j>1).
\]
Direct substitution for $k=0,1,2,3$ gives $0<x_1<1$ and all $x_j>0$. Hence, precisely the remaining negative peripheral signs are mismatches. They are symmetric, so either rule corrects them one at a time. For $k=4$, $x_1<0$ and the peripheral coordinates stay positive and coordinate
$1$ is the only mismatch. Finally, its correction terminates at $x^*$.
\end{proof}

\begin{remark}
Block-diagonal replication of \eqref{eq:5e} gives $6n/5$ flips for $n$ divisible by five,
from the replicated arbitrary initialization. Each block has the displayed
local sequence regardless of interleaving with other blocks. This disproves
a universal unit coefficient on the full $\eta=0.49$ class, but supplies
neither a quadratic lower bound nor the optimal leading constant.
The sufficient norm thresholds in Subsection~\ref{sec:unit} are likewise not
claimed to be sharp thresholds.
\end{remark}

\subsection{Exponential separation under arbitrarily strong contraction}\label{sec:lower}
The linear bounds above rely on the greedy choice of a coordinate. To
assess the role of this choice, we compare the greedy rules with
least-index selection on the same instance and from the same natural
initialization. The construction below gives an exponential trajectory
for the least-index rule and a single flip for both greedy rules, even
when both inverse norms are arbitrarily small and the rational input has
polynomial encoding length.

\begin{theorem}\label{thm:lower}
For every fixed $0<\eta<1/2$ and every integer $n\ge2$, there are
$A\in \mathbb{Q}^{n\times n}$ and $b\in \mathbb{Q}^n$ with total binary encoding length $\mathcal{O}(n^2)$ (where the constant may
depend on $\eta$) such that
\begin{equation*}
 \max\{\|A^{-1}\|_1,\|A^{-1}\|_\infty\}<\eta.
\end{equation*}
Starting at $z=\sgn(A^{-1}b)$, the least-index rule executes exactly
$2^{n-1}-1$ flips. On the same instance and from the same initial signature,
both greedy rules execute exactly one flip.
\end{theorem}

\begin{proof}
Choose a rational number
\begin{equation}\label{eq:tchoice}
 0 < t <\min\{\eta/2,1/16\}.
\end{equation}
Let $N$ be the upper shift matrix, $N_{i,i+1}=1$, and set
\begin{equation}\label{eq:ea}
 A^{-1}=t(I+N), \qquad k=n-1.
\end{equation}
Then $\|A^{-1}\|_1=\|A^{-1}\|_\infty=2t<\eta$.
Define
\begin{equation}\label{eq:construction}
 q=\frac{t}{2(1+t)},\quad L=\frac{t}{2(1-t^2)},\quad
 \varepsilon=Lq^k,\quad
 d=(\underbrace{\varepsilon,\ldots,\varepsilon}_{n-2},
       -t/[2(1-t)],1)^T,\quad b=Ad.
\end{equation}
Thus $z^{(0)} =(1,\ldots,1,-1,1)^T$. The candidate equations are
\begin{align}
 (1-tz_i)x_i&=d_i+tz_{i+1}x_{i+1},&&i<n,\label{eq:recursion}\\
 (1-tz_n)x_n&=1.\nonumber
\end{align}
All their diagonal coefficients are positive, so Assumption~\ref{assu:regular} holds.

Whenever $z_n=1$, we have $x_n=1/(1-t)>0$. This coordinate is initially
matched and can never be selected. For every signature with $z_n=1$,
\[
 x_k=\frac{t}{2(1-t)(1-tz_k)}\ge L>0.
\]
We claim that, for all such signatures,
\begin{equation}\label{eq:suffix}
 \sgn(x_i(z))=\prod_{j=i+1}^k z_j,\qquad
 |x_i(z)|\ge L_i:=Lq^{k-i},\qquad 1\le i\le k.
\end{equation}
The claim is true for $i=k$, using the empty-product convention.
If it holds at $i+1$, then
\[
 \varepsilon\le L_i,\qquad
 t|x_{i+1}(z)|\ge tL_{i+1}=2(1+t)L_i>\varepsilon.
\]
Equation~\eqref{eq:recursion} therefore gives
$\sgn(x_i(z))=z_{i+1}\sgn(x_{i+1}(z))$ and
\[
 |x_i(z)|\ge\frac{tL_{i+1}-\varepsilon}{1+t}
       \ge\frac{2(1+t)-1}{1+t}L_i\ge L_i.
\]
Backward induction proves~\eqref{eq:suffix}.

Define
\[
 u_i(z)=\frac{1-\prod_{j=i}^kz_j}{2}\in\{0,1\},\qquad
 B(z)=\sum_{i=1}^k2^{i-1}u_i(z).
\]
By~\eqref{eq:suffix}, $i\le k$ is mismatched exactly when $u_i=1$.
If the least-index rule selects $i$, then
$u_1=\cdots=u_{i-1}=0$ and $u_i=1$. Reversing $z_i$ complements exactly
$u_1,\ldots,u_i$. Hence
\[
 B(z^{+})-B(z)=\sum_{j=1}^{i-1}2^{j-1}-2^{i-1}=-1.
\]
At initialization every $u_i$ equals one, so $B(z^{(0)})=2^k-1$.
Termination is equivalent to $B(z)=0$. This proves the exact flip count.

At $z^{(0)}$, all indices $1,\ldots,k$ are mismatched. Moreover, for $i<k$,
\begin{equation}\label{eq:initialratio}
 |x_i(z^{(0)})|=\frac{t|x_{i+1}(z^{(0)})|-\varepsilon}{1-t}
 <\frac{t}{1-t}|x_{i+1}(z^{(0)})|.
\end{equation}
Thus the Gauss--Southwell rule uniquely selects $k$. That flip produces the all-positive
signature, which is compatible by~\eqref{eq:suffix}. Hence, the Gauss--Southwell rule stops after one flip.

For the steepest single-flip rule, recall
\[
 \Delta F_i =2|x_i(z)|p_i(z^+)
 \quad(i\in\cV_z),\qquad
 p(z)=(I-A^{-T}Z)^{-1}\one.
\]
The norm certificate in Lemma~\ref{lem:norm} yields
\[
 \alpha=\frac{1-4t}{1-2t}\le p_i(z)\le
 \beta=\frac1{1-2t},\qquad \frac\alpha\beta=1-4t.
\]
Since $t<1/16$ (due to \eqref{eq:tchoice}), we have $t/(1-t)<1-4t$.
For every $i<k$ and $z = z^{(0)}$, equation~\eqref{eq:initialratio} implies
\[
 \Delta F_i
 \le2\beta|x_i(z^{(0)})|
 <2\alpha|x_{i+1}(z^{(0)})|
 \le \Delta F_{i+1}.
\]
The steepest single-flip rule also uniquely chooses $k$ and terminates after one flip.

Finally,
\[
 A=t^{-1}(I+N)^{-1}=t^{-1}\sum_{r=0}^{n-1}(-N)^r.
\]
For fixed rational $t$, its entries have constant encoding length, while
$\varepsilon=Lq^{n-1}$ has $\mathcal{O}(n)$ encoding length.
Each entry of $b=Ad$ is a sum of at most $n$ terms with a common denominator
of $\mathcal{O}(n)$ bits, and hence has $\mathcal{O}(n)$ bits. The total encoding length is
$\mathcal{O}(n^2)$.
\end{proof}

\begin{remark}[What the lower bound measures]
The example has a unique solution because the map $x\mapsto A^{-1}|x|+d$ is a
contraction in the infinity norm. The exponential trajectory is therefore
not caused by nonuniqueness or proximity to loss of contraction.
The right-hand side does contain exponentially small entries and their binary
encoding length is only linear in $n$. This is a worst-case exact-arithmetic
result, not a robustness or smoothed-complexity claim.
The restriction $t<1/16$ is used only to obtain the simultaneous one-step
statement for the steepest single-flip rule by a simple uniform estimate.
\end{remark}

\section{A single-flip consequence of signed elimination}\label{sec:rr}
The preceding section shows that coordinate selection can strongly affect the number of flips. This section records an existing-mechanism comparison, not a faster version
of the Gauss--Southwell or the steepest ascent. The norm-based sign-identification and
elimination ideas come from Radons~\cite{rado2016} and
Radons--Rump~\cite{raru2022}. We give elementary proofs for the strict
$1/2$ infinity-norm regime and make their implication for an admissible
single-flip path explicit. We do not assert here that the other structural
hypotheses in those papers follow from this particular proof.

\subsection{A sign certificate valid for every remaining signature}
The following statement is stronger than certifying a coordinate of the
nonlinear solution \cite{rado2016,raru2022}. It certifies that coordinate in every candidate solve.

\begin{lemma}\label{lem:signcertificate}
Let $E\in\R^{n \times n}$, $\|E\|_\infty\le r < 1/2$, and $f \ne 0$.
For any signature $s\in\cZ$, let
$u=(I-E\Diag(s))^{-1}f$.
If $i\in\argmax_j|f_j|$, then $u_i\ne0$ and
\begin{equation}\label{eq:certsign}
 \sgn(u_i)=\sgn(f_i).
\end{equation}
\end{lemma}

\begin{proof}
Set $H=(I-E\Diag(s))^{-1}$.
The Neumann series gives
\[
 \|H-I\|_\infty\le\frac r{1-r}<1.
\]
For each row,
\[
 H_{ii}-\sum_{j\ne i}|H_{ij}|
 \ge1-|H_{ii}-1|-\sum_{j\ne i}|H_{ij}|
 \ge1-\frac r{1-r}>0.
\]
In particular $H_{ii}>0$. If $|f_i|=\|f\|_\infty>0$, the diagonal
term $H_{ii}f_i$ strictly dominates the sum of all other terms in
$u_i=\sum_{j}H_{ij}f_j$. This proves~\eqref{eq:certsign}.
\end{proof}

Suppose $\sigma=\sgn(f_i)$ is fixed as the signature at coordinate $i$.
With $J=[n]\setminus\{i\}$, eliminate $u_i$ by setting
\begin{equation}\label{eq:elimination}
 E^+=E_{JJ}+\frac{\sigma}{1-E_{ii}\sigma}E_{Ji}E_{iJ},
 \qquad
 f^+=f_J+\frac{\sigma f_i}{1-E_{ii}\sigma}E_{Ji}.
\end{equation}
For every extension of the fixed sign, these are exactly the remaining
candidate equations. No assumption about the signs of the remaining
candidate coordinates is required.

\begin{lemma}[Preservation under elimination]\label{lem:closure}
If $\|E\|_\infty\le r<1$, then for either $\sigma\in\{-1,1\}$,
$1-E_{ii}\sigma>0$ and the matrix $E^+$ in~\eqref{eq:elimination}
satisfies $\|E^+\|_\infty\le r$.
\end{lemma}

\begin{proof}
The denominator is at least $1-|E_{ii}|>0$.
For any row $j\ne i$,
\begin{align*}
 \sum_{\ell\ne i}|E^+_{j\ell}|
 &\le\sum_{\ell\ne i}|E_{j\ell}|
   +\frac{|E_{ji}|}{1-|E_{ii}|}
         \sum_{\ell\ne i}|E_{i\ell}|\\
 &\le r-|E_{ji}|+|E_{ji}|\frac{r-|E_{ii}|}{1-|E_{ii}|}
 \le r.
\end{align*}
The proof is completed.
\end{proof}

\subsection{The elimination-guided single-flip rule}\label{subsec:eg}
Under the strict $1/2$ infinity-norm condition, a maximal component of the (reduced) right-hand side determines a coordinate sign shared by every remaining candidate solve. Eliminating this coordinate with its certified sign preserves the norm bound, allowing the certification argument to be applied recursively. These properties allow certified signs to be fixed permanently and provide the basis for the single-flip method described in Algorithm~\ref{alg:rr}.

\begin{algorithm}[h]
\caption{Elimination-guided single flip method}
\label{alg:rr}
\begin{algorithmic}[1]
\REQUIRE Input $A, b$ with $\|A^{-1}\|_\infty<1/2$ and any $z\in\cZ$.

\STATE Set $J=[n]$, $(E,f)=(A^{-1}, A^{-1}b)$ and compute $x=(A - \Diag(z))^{-1}b$.

\WHILE{$J\ne\varnothing$ and $f\ne 0$}

    \STATE Select $i\in J$ maximizing $|f_i|$ and set $\sigma=\sgn(f_i)$.

    \STATE If $z_i\ne\sigma$, flip $z_i$ and update the full candidate
$x=(A - \Diag(z))^{-1}b$.

    \STATE Certify coordinate $i$ and keep its signature fixed henceforth. Delete $i$ from $J$
 and replace $(E,f)$ by~\eqref{eq:elimination}.

\ENDWHILE
\RETURN AVE solution $x$.
\end{algorithmic}
\end{algorithm}

Algorithm~\ref{alg:rr} maintains a set of certified coordinates, whose signs will never change,
and a reduced normalized system $(E,f)$ on the remaining indices $J$.
The reduced matrices are used for selecting a coordinate. Actual sign flips
and candidate updates are performed in the full system. The rule is not obtained by maximizing the original candidate magnitude or the potential increment. It uses a recursively
updated right-hand side and permanently certified signs. The sign certificate ensures that each elimination stage either leaves an already correct sign unchanged or corrects it by an admissible mismatch flip. In both cases, the certified coordinate remains matched throughout all subsequent candidate solves. The following theorem uses this permanence to establish the exact flip count and bound the number of full candidate solves.

\begin{theorem}[At most one flip per coordinate]\label{thm:rr}
Algorithm~\ref{alg:rr} is well defined and returns the unique solution
of AVE~\eqref{eq:ave} from any initial signature $z^{(0)}\in \cZ$.
Every actual flip is a strict mismatch flip in the full system.
The number of actual flips is exactly ${\rm card}(D(z^{(0)})) \le s\le n$, and the number of full
candidate-system solves is at most $n+1$, including initialization.
There are at most $n$ elimination stages.
\end{theorem}
\begin{proof}
Uniqueness follows, for example, from contraction of $x\mapsto A^{-1}|x|+d$
in the infinity norm. For Algorithm~\ref{alg:rr}, we maintain the following invariant: the full candidates for every
extension of the already fixed signature have matching signs on all
certified coordinates, and their remaining coordinates solve the current
reduced candidate equations. It is vacuous initially.

Lemma~\ref{lem:closure} preserves the bound $\|E\|_\infty<1/2$.
If $f\ne0$, Lemma~\ref{lem:signcertificate} shows that the selected
coordinate has the nonzero sign $\sigma$ for every choice of the remaining
signature, including the current one. Thus $z_i\ne\sigma$ is exactly a
strict mismatch at that coordinate, and the prescribed flip is admissible.
After this possible flip, $z_i=\sigma$. The certificate holds for all
future extensions, so the coordinate will remain matched.
Ordinary elimination gives~\eqref{eq:elimination} for the remaining
candidate equations, preserving the invariant.

If all coordinates have been eliminated, the full candidate is compatible.
If instead the reduced right-hand side becomes zero, then all remaining
candidate coordinates vanish, since $I-E\Diag(z_J)$ is nonsingular.
Those coordinates are compatible with either sign and the certified ones
remain matched by the invariant. Thus this stopping case is also correct. Each certified coordinate has a nonzero true solution component and is
never flipped after certification. Each stage deletes a distinct coordinate
and performs at most one flip. Stages without a flip need no new full candidate solve. An initially correct sign is never changed while
each initially wrong nonzero sign must be corrected before termination.
This proves the exact count ${\rm card}(D(z^{(0)}))$. Initialization plus the actual flips gives the stated solve count.
\end{proof}

\begin{remark}[Computational interpretation]
The $n$-flip bound includes auxiliary elimination work.
It is not a bound of $n$ arithmetic operations. Updating the reduced system
costs $\mathcal{O}({\rm card(J)}^2)$ operations per stage. Maintaining a full inverse for the
candidate system by rank-one updates costs $\mathcal{O}(n^2)$ operations per flip.
Together with initialization, this gives an $\mathcal{O}(n^3)$ dense arithmetic
implementation.
\end{remark}

The strict infinity-norm condition is sufficient for the elimination
rule, but it depends on the scaling of the variables. Positive diagonal
scaling extends the guarantee to the spectral condition
$\rho(|A^{-1}|)<1/2$. Because this scaling preserves coordinate signs,
the flips selected in the scaled system remain admissible in the
original system.

\begin{corollary}\label{cor:rrspectral}
If $\rho(|A^{-1}|)<1/2$, an elimination-guided single-flip
rule exists with at most $n$ flips from every initial signature.
In particular this applies whenever $\|A^{-1}\|_1<1/2$.
\end{corollary}

\begin{proof}
Choose $\rho(|A^{-1}|)<\tau<1/2$ and define
\[
 v=(I-|A^{-1}|/\tau)^{-1}\one>0,\qquad D=\Diag(v).
\]
Since $|A^{-1}|v=\tau(v-\one)<\tau v$, the matrix
$\widetilde A^{-1}=D^{-1}A^{-1}D$ satisfies
\[
 \|\widetilde A^{-1}\|_\infty
 =\max_i\frac{(|A^{-1}|v)_i}{v_i}<\tau<1/2.
\]
The change of variables $x=Du$ gives
$u-\widetilde A^{-1}|u|=D^{-1}d$.
Run Algorithm~\ref{alg:rr} on this scaled equation. For every signature,
its candidate is $u(z)=D^{-1}x(z)$. Moreover, positive diagonal scaling preserves
both the signature and the mismatch set. Thus its at most $n$ flips
are valid in the original system. In general this selection rule uses
the scaling through the reduced data. Since
$\rho(|A^{-1}|)\le\|A^{-1}\|_1$, the inverse $1$-norm case follows.
\end{proof}

\subsection{Cyclic scanning without explicit certificates}\label{sec:cyclic}
The elimination-guided rule uses reduced data to select coordinates,
and its spectral extension may also require a scaling. There is a particularly simple way to exploit the sign-certification
structure without computing a reduced system. Fix an ordering of the
coordinates, scan them cyclically, and flip a coordinate exactly when it
is currently mismatched. After a flip, advance the pointer to the next
coordinate instead of restarting the search at the first coordinate.
For concreteness the order used in Algorithm~\ref{alg:cs} is $1,\ldots,n$.

\begin{algorithm}[h]
\caption{Cyclic mismatch scan method}
\label{alg:cs}
\begin{algorithmic}[1]
\REQUIRE Input $A,b$ and any $z\in\cZ$. Compute $x=(A-\Diag(z))^{-1}b$.

\FOR{at most $n$ sweeps}

    \STATE Set a flag $\mathrm{changed}=\mathrm{false}$.

    \STATE For $i=1,\ldots,n$, inspect $z_i x_i$. If it is negative,
flip $z_i$, update the full candidate $x=x(z)$, and set
$\mathrm{changed}=\mathrm{true}$. In either case advance the pointer to $i+1$.

    \STATE If $\mathrm{changed}=\mathrm{false}$, return $x$.
\ENDFOR
\RETURN After the $n$th sweep, return $x$.
\end{algorithmic}
\end{algorithm}

The cap in Algorithm~\ref{alg:cs} is justified by the next theorem and is intended
for its stated matrix class. A sweep with no flips also certifies termination
without any global rescan, since the candidate is then unchanged throughout
all $n$ inspections. Without a known sweep cap or a separate compatibility
test, a generic implementation that insists on a final no-flip sweep may
use one additional sweep after it has already reached a solution.

\begin{theorem}[Cyclic scanning bound]\label{thm:cyclic}
Suppose that $\|A^{-1}\|_\infty<1/2$. Starting from any initial signature, cyclic mismatch scanning reaches a compatible signature in at most $n$ sweeps. Consequently, Algorithm~\ref{alg:cs} returns the unique solution using at most $n^2$ coordinate inspections and $n(n+1)/2$ flips, without computing elimination data or polarity factors. These guarantees hold for any fixed coordinate order.
\end{theorem}

\begin{proof}
Construct, only for the proof, a sequence of sign certificates by applying
Lemma~\ref{lem:signcertificate} and the elimination
formula~\eqref{eq:elimination} recursively. Let the selected coordinates
be $i_1,\ldots,i_q$, with certified signs $\sigma_1,\ldots,\sigma_q$.
Stop when all coordinates have been eliminated or the remaining right-hand
side is zero. Here $q\le n$, and $q=0$ is allowed. Crucially, the
certificate sequence is defined from the reduced data, not from the
actual scanning trajectory. Lemma~\ref{lem:signcertificate} implies that whenever the signature
agrees with the first $r-1$ certificates, its candidate coordinate $i_r$
has the nonzero sign $\sigma_r$, regardless of the other signatures.
If all $q$ certificates hold, every remaining candidate coordinate is zero
in the early-stop case, so the full candidate is compatible.

We prove by induction that, by the end of sweep $r$, the first
$\min\{r,q\}$ certificates hold permanently. Initially there are no
fixed certificates. Suppose the first $r-1$ certificates hold permanently
at the start of sweep $r\le q$. Their coordinates cannot be flipped by
any mismatch rule: their candidate signs remain certified under every
extension of the remaining signature. Throughout this sweep, coordinate
$i_r$ has candidate sign $\sigma_r$. When the pointer visits it, its
signature either already equals $\sigma_r$ or is corrected by a mismatch
flip. Thereafter it also remains permanently matched. This proves the
induction, without requiring the scan order to agree with the certificate
order.

After at most $q\le n$ sweeps all certificates hold, and the candidate
solves the AVE. There are at most $n$ inspections per sweep.
At the start of sweep $r\le q$, at least $r-1$ coordinates are permanently
matched and each of the other coordinates can be flipped at most once in
that sweep. Hence the total number of flips is bounded by
\[
 \sum_{r=1}^{q}(n-r+1)\le\frac{n(n+1)}2.
\]
Any subsequent no-flip sweep adds no flips. Algorithm~3 either exits on
such a sweep or uses its justified cap at $n$ sweeps, so its total
inspection count is at most $n^2$.
\end{proof}

Cyclic scanning depends only on coordinate signs, which positive diagonal scaling preserves. With the same initial signature and scan order, the scaled and original systems therefore produce identical decisions at every inspection. This allows the preceding bounds to extend to the spectral condition below, with the scaling used only in the proof.

\begin{corollary}[Scale-free spectral guarantee]\label{cor:cyclicspectral}
The conclusion of Theorem~\ref{thm:cyclic} holds under
$\rho(|A^{-1}|)<1/2$, without computing a diagonal scaling. In particular,
it holds under $\|A^{-1}\|_1<1/2$.
\end{corollary}

\begin{proof}
Use the positive diagonal scaling of Corollary~\ref{cor:rrspectral} to
obtain $\|D^{-1}A^{-1}D\|_\infty<1/2$. At every common signature,
$u(z)=D^{-1}x(z)$, so the scalar tests $z_i u_i(z)<0$ and
$z_i x_i(z)<0$ are identical. Starting with the same signature and pointer,
the scaled and unscaled cyclic algorithms make the same decisions at every
inspection, by induction. The norm-based Theorem~\ref{thm:cyclic} therefore applies to the
unscaled algorithm with exactly the same inspection and flip counts.
\end{proof}

\begin{remark}[Why restarting changes the complexity]
The certificate argument requires a bounded delay before a safe coordinate
is inspected. Cyclic scanning provides this through its advancing pointer.
The least-index rule restarts at coordinate one after each flip and can
repeatedly postpone a safe coordinate while lower-index coordinates become
mismatched again. Theorem~\ref{thm:lower} realizes precisely such a delay.
Thus replacing a restarted scan by a cyclic pointer changes the worst-case
guarantee from exponential to polynomial on the common small-norm class.
The proof is a consequence of the signed-elimination certificates of
Subsection~\ref{subsec:eg}, rather than of the polarity contraction argument.
It does not establish the cyclic bound under Assumption~\ref{ass:polarity} alone. More generally, the same induction gives at most $nt$ inspections
for an admissible schedule in which every block of $t$ inspections visits
all coordinates. Only the cyclic specialization is used here.
\end{remark}

\subsection{The randomized \texorpdfstring{$m$}{m}-flip rule}
Accordingly, cyclic scanning is simple to implement, requires no explicit computation of elimination certificates, and provides a deterministic polynomial guarantee. However, its drawbacks are that it may perform many inspections and flips and lacks adaptivity in coordinate selection.
To circumvent these drawbacks, we introduce the following randomized $m$-flip variant.

\par\addvspace{\ruleboxskip}
\noindent\makebox[\linewidth][c]{%
  \fbox{%
    \begin{minipage}{0.92\linewidth}
      \textbf{Randomized $m$-flip rule:}\\
       $S\subseteq \mathcal{V}_z$ uniformly at random with ${\rm card}(S) = \min\{m,{\rm card}(\mathcal{V}_z)\}$, where $m\in [n]$ is given.
    \end{minipage}%
  }%
}
\par\addvspace{\ruleboxskip}

The certificate-prefix induction remains unchanged, since a certified coordinate stays matched
under every extension of the remaining signature. The difference is that
the next uncertified coordinate is now reached probabilistically. Specifically, we have

\begin{theorem}[Uniform expected iteration bound]\label{thm:rrmf-bound}
Suppose that $\|A^{-1}\|_\infty < 1/2$, and let $m \in [n]$ be an integer.
For every $b \in \R^n$ and every initial signature, the randomized
$m$-flip rule terminates almost surely at the unique solution
of AVE~\eqref{eq:ave}. Its number of batch updates satisfies
\begin{equation}\label{eq:rrmf-bound}
 \mathbb{E}[K]\le B_{n,m}:=  \sum_{r=1}^n\max\left\{1,\frac rm\right\}
 =\frac{n(n+1)}{2m}+\frac{m-1}{2}.
\end{equation}
\end{theorem}

\begin{proof}
Analogous to the proof of Theorem~\ref{thm:cyclic}, we can construct
distinct indices $i_1,\ldots,i_q$ and signs
$\sigma_1,\ldots,\sigma_q$, with $0\le q\le n$. This construction depends
only on $A^{-1},A^{-1}b$ and the fixed convention, not on the random trajectory.
Exactness of the reduced candidate equations implies that, for each
$j\le q$,
\begin{equation}\label{eq:rrmf-prefix}
 \left.
 \begin{gathered}
 z_{i_h}=\sigma_h\quad(1\le h<j)
 \end{gathered}
 \right\}
 \quad\Longrightarrow\quad
 x_{i_j}(z)\ne0,\quad\sgn(x_{i_j}(z))=\sigma_j,
\end{equation}
regardless of all other signature entries. If all $q$ certificates
hold, the full candidate is compatible. Indeed, all certified
coordinates are matched, and any uneliminated coordinates have zero
candidate values because their reduced right-hand side is zero and
their reduced candidate matrix is nonsingular. This also covers $q=0$,
in which case $d=0$ and $K=0$.

For a nonterminal state at iterate $k$, define its correct certificate
prefix length by
\[
 L_k=\max\{\ell\in\{0,\ldots,q\}:
 z_{i_h}^{(k)}=\sigma_h\text{ for every }1\le h\le\ell\}.
\]
By~\eqref{eq:rrmf-prefix}, the coordinates in this prefix are matched
for every extension of the remaining signature. They cannot belong to
$\cV_{z^{(k)}}$ and hence cannot be flipped. Thus $L_k$ never decreases, and
$L_k=q$ implies termination.

Suppose $L_k=\ell<q$ and set $r=n-\ell$. Maximality of $\ell$ and
\eqref{eq:rrmf-prefix} imply $i_{\ell+1}\in \cV_{z^{(k)}}$. Also, the first
$\ell$ certified coordinates are excluded from $\cV_{z^{(k)}}$, so
$1\le {\rm card}(\cV_{z^{(k)}}) \le r$. Writing $\mathcal F_k$ for the history before drawing
$S \subseteq \cV_{z^{(k)}}$, uniform sampling without replacement gives
\begin{align}
P(i_{\ell+1}\in S\mid\mathcal F_k)
 &=\frac{\min\{m,{\rm card}(\cV_{z^{(k)}})\}}{{\rm card}(\cV_{z^{(k)}})}\notag\\
 &=\min\left\{1,\frac{m}{{\rm card}(\cV_{z^{(k)}})}\right\}
 \ge\min\left\{1,\frac mr\right\}=:p_r.
 \label{eq:rrmf-success}
\end{align}
If this coordinate is selected, it receives its certified sign, and
the existing prefix is preserved. Therefore $L_{k+1}\ge\ell+1$.
Simultaneous flips of the other selected coordinates do not affect
this conclusion, by the uniformity in~\eqref{eq:rrmf-prefix}.

Let $W_\ell$ count the nonterminal batch updates with $L_k =\ell$,
and set $W_\ell=0$ if that stage is skipped. Conditional on entering
the stage, the probability of remaining there after $h$ updates is at
most $(1-p_{r})^h$. This follows by conditioning successively on
the history and applying~\eqref{eq:rrmf-success} at each update and
independence of the successive success events is unnecessary.
Consequently, the tail-sum formula gives
\[
 \mathbb{E}[W_\ell]\le\sum_{h=0}^{\infty}(1-p_{r})^h
 =\frac{1}{p_{r}}
 =\max\left\{1,\frac{r}{m}\right\}.
\]
Every nonterminal update belongs to one such stage. Hence, also allowing
extended nonnegative values before proving finiteness,
\begin{align*}
 \mathbb{E}[K]
 &=\sum_{\ell=0}^{q-1}\mathbb{E}[W_\ell]\\
 &\le\sum_{\ell=0}^{q-1}\max\left\{1,\frac{n-\ell}{m}\right\}
 \le\sum_{r=1}^n\max\left\{1,\frac rm\right\}.
\end{align*}
The last sum is finite, so $K<\infty$ almost surely. Finally,
\[
 \sum_{r=1}^n\max\left\{1,\frac rm\right\}
 =m+\frac1m\sum_{r=m+1}^n r
 =\frac{n(n+1)}{2m}+\frac{m-1}{2},
\]
which proves~\eqref{eq:rrmf-bound}.
\end{proof}

Similar to Corollary~\ref{cor:cyclicspectral}, we have

\begin{corollary}\label{cor:rrmf-scaling}
The conclusion of Theorem~\ref{thm:rrmf-bound} holds under
$\rho(|A^{-1}|)<1/2$, without computing a positive diagonal scaling. In particular,
it holds under $\|A^{-1}\|_1<1/2$.
\end{corollary}

\begin{remark}[Meaning and scope of the bound]
The estimate is uniform over all admissible inputs and initial
signatures, so it is a worst-case expected iteration bound. For
$1\le m\le n$, $B_{n,m}= \mathcal{O}(n^2/m)$. At $m=1$, the randomized single-flip rule has
$\mathbb{E} [K]\le n(n+1)/2$ while, at $m=n$, the certificate argument gives the deterministic
bound $K \le n$ for the full flip rule.

The quantity $K$ counts batch updates, not individual sign changes or
arithmetic operations. Translating it into an arithmetic-work bound
requires the cost of locating mismatches, sampling, and updating the
candidate. Uniform sampling is from the current mismatch set. The
probability estimate~\eqref{eq:rrmf-success} does not automatically apply
to sampling from all $n$ coordinates. No assertion is made that the
coefficient in~\eqref{eq:rrmf-bound} is sharp.
\end{remark}

\section{Comparison, arithmetic cost, and scope}\label{sec:discussion}
This section compares the preceding rules under their respective
assumptions, relates flip and inspection counts to arithmetic work, and
clarifies the scope of the conclusions. We first give a cycling example that separates solvability from termination of a particular selection rule, then summarize the
guarantees. We conclude with the costs of dense implementations and the questions left open by the analysis.

\subsection{A six-cycle of the steepest single-flip rule}\label{subsec:steepestcycle}
The preceding guarantees rely on conditions that control the selected
flips. The following example shows that $\|A^{-1}\|_1<1$ alone does not
guarantee finite termination of steepest single flipping, even though it
ensures a unique solution and nonsingular candidate systems. The failure
occurs from natural initialization, with no zero candidate components or
tied choices along the trajectory. Here steepest single-flip rule selects the largest
potential increment among current mismatches, including when all such
increments are negative.

\begin{example}\label{exam:sc}
Let
\[
A^{-1}
=\frac{1}{16}
\begin{pmatrix}
0&-12&0&0&4\\
0&0&-12&0&0\\
-12&0&0&0&0\\
-2&-2&-2&15&0\\
0&0&0&0&8
\end{pmatrix},
\qquad
d=A^{-1}b=
\begin{pmatrix}
-1\\ 1\\ 1\\ 1/4\\ 4
\end{pmatrix}.
\]
Consequently, $\lVert A^{-1} \rVert_1=\frac{15}{16}<1$. Every candidate system is nonsingular, and  AVE~\eqref{eq:ave} has the unique solution $x^*=(4/7,4/7,4/7,4/7,8)^{T}$. The natural initialization is $z^{(0)}=\sgn(d)=(-1,1,1,1,1)^{T}$. Starting from $z^{(0)}$, the steepest single-flip rule enters a six-cycle.

Indeed, for a signature $z\in\{-1,1\}^5$, define $x = x(z)=\bigl(I-A^{-1}\Diag(z)\bigr)^{-1}d$. For every signature with \(z_4=z_5=1\), one has \(x_5=8 >0\) and \(x_4>0\). Hence coordinates \(4\) and \(5\) are never mismatched and can never be selected. It therefore suffices to consider states satisfying $z_4=z_5=1$.  For every such
state, $x_5=8$ and the first three candidate equations reduce to
$$
x_1=1-\frac{3}{4}z_2x_2,
\qquad
x_2=1-\frac{3}{4}z_3x_3,
\qquad
x_3=1-\frac{3}{4}z_1x_1.
$$
Set $q(z)=z_1x_1+z_2x_2+z_3x_3$. The fourth equation then yields $x_4=4-2q(z)$. Then we have
$$
F(z)=12-q(z).
$$

Table~\ref{tab:cycle} suppresses the fixed signature components
$z_4=z_5=1$ and the fixed candidate component $x_5=8$. At every state in Table~\ref{tab:cycle}, both $x_4$ and $x_5$ are
positive.  Hence, coordinates $4$ and $5$ remain sign-consistent
throughout the cycle.

\begin{table}[htbp]
\centering
\renewcommand{\arraystretch}{1.25}
\caption{The six states visited by the steepest single-flip rule.}
\label{tab:cycle}
\begin{tabular}{@{}c c c c@{}}
\toprule
First three signs
& $(x_1,x_2,x_3,x_4)$
& Mismatch set
& Steepest choice \\
\midrule
$(-1,1,1)$
& $(52,-20,76,140)/37$
& $\{1,2\}$
& $2$ \\
$(-1,-1,1)$
& $(76,-20,148,180)/91$
& $\{1\}$
& $1$ \\
$(1,-1,1)$
& $(76,52,-20,140)/37$
& $\{2,3\}$
& $3$ \\
$(1,-1,-1)$
& $(148,76,-20,180)/91$
& $\{2\}$
& $2$ \\
$(1,1,-1)$
& $(-20,76,52,140)/37$
& $\{1,3\}$
& $1$ \\
$(-1,1,-1)$
& $(-20,148,76,180)/91$
& $\{3\}$
& $3$ \\
\bottomrule
\end{tabular}
\end{table}

We verify the first step explicitly.  At the initial signature, $F(z^{(0)})=440/37$.
Flipping coordinate $1$ leads directly to the all-positive signature
and hence to the unique solution and the corresponding potential value is $F(z^+)= {72}/{7}$. By contrast, flipping coordinate $2$ leads to the second row of
Table~\ref{tab:cycle}, with $F(z^+)= {1000}/{91}$.
Since $\Delta F_2 =  - {3040}/{3367}> - {416}/{259}  = \Delta F_1$, the steepest rule strictly selects coordinate $2$. At the second state, coordinate $1$ is the only mismatched coordinate,
so the next flip is forced.  The first three candidate equations are
cyclically symmetric and then the remaining four transitions therefore follow
by the same argument. Every selection is unique, and after six steps the method returns to
the complete initial signature. It therefore cycles indefinitely.

For this example, both Gauss-Southwell and least index rules terminate in one step. Indeed, at the initial state, the mismatch set is $\{1,2\}$ and
\[
\lvert x_1(z^{(0)})\rvert= \frac{52}{37}
>
\frac{20}{37}=\lvert x_2(z^{(0)})\rvert.
\]
The  Gauss-Southwell rule therefore selects coordinate $1$.  The least index rule also
selects coordinate $1$, because it is the mismatched coordinate with
the smallest index.  After this flip, the signature is all-positive,
and the resulting candidate vector is precisely $x^*$.

This example lies outside the preceding polarity-based termination
regime. Both admissible increments at the initial state are negative,
so uniform polarity positivity cannot hold, and
$\|A^{-1}\|_\infty=21/16>1$ as well. Maximizing the increment therefore
does not ensure that the potential increases. The six-cycle shows why
uniqueness and Assumption~\ref{assu:regular} alone cannot replace the
assumptions used in the termination proofs.
\end{example}

\subsection{Comparison of flip and inspection bounds and arithmetic cost}
Table~\ref{tab:comparison} collects the guarantees together with their
assumptions. Exact correction counts, dimension-dependent flip bounds,
and inspection counts describe different aspects of an algorithm and
must be read with the corresponding parameter restrictions.

\begin{table}[ht]
\centering\small
\caption{Single-flip guarantees in exact arithmetic. Constants are uniform
only when the displayed conditioning parameters are bounded.}\label{tab:comparison}
\setlength{\tabcolsep}{4pt}
\begin{tabularx}{\textwidth}{@{}ll>{\raggedright\arraybackslash}X@{}}
\toprule
Rule & Assumption & Flip guarantee\\
\midrule
GS / steepest & Polarity ratio $\kappa$, $r=\|A^{-1}\|_\infty<1$ &
$\lfloor s\kappa(1+r)/(1-r)\rfloor$\\
GS / steepest & $\|A^{-1}\|_1 \le\eta_1<1/2$ and $r\le\eta_2<1$, fixed $\eta_1,\eta_2$ &
$\mathcal{O}_{\eta_1,\eta_2}(n)$ worst-case order\\
GS & $r <1/3$ & Exactly ${\rm card}(D(z^{(0)}))\le n$\\
Steepest & Polarity and $2\kappa r<1-r$ & Exactly ${\rm card}(D(z^{(0)}))\le n$\\
GS / steepest & Diagonal-sensitive certificates & Exactly ${\rm card}(D(z^{(0)}))\le n$\\
Elimination-guided & $\rho(|A^{-1}|)<1/2$, with scaling if needed & Exactly ${\rm card}(D(z^{(0)}))\le n$\\
Cyclic scan & $\rho(|A^{-1}|)<1/2$ & $n^2$ inspections; $n(n+1)/2$ flips\\
Randomized $1$-flip & $\rho(|A^{-1}|)<1/2$ & Expected $n(n+1)/2$ flips\\
Least-index & Arbitrarily small simultaneous norms & $2^{n-1}-1$ on the construction\\
Steepest & $\|A^{-1}\|_1<1$ alone & Can cycle; see Example~\ref{exam:sc}\\
\bottomrule
\end{tabularx}
\end{table}

The $\mathcal{O}_{\eta_1,\eta_2}(n)$ statement concerns dimension dependence for fixed margins.
It does not assert optimality of the coefficient in~\eqref{eq:gsst}.
The exact correction count is optimal for a fixed starting signature because
each initially wrong nonzero coordinate must be reversed at least once.
Proposition~\ref{prop:six} rules out extending a unit coefficient to the
whole common small-norm class for arbitrary initialization. Full-flip is a comparison method, not a single-flip rule. Its linear
iteration guarantee is prior work~\cite{raru2022}. The new
result here places the two greedy single-flip rules at the same dimension
order under the stated common assumptions. This is not an assertion that
all three have the same practical cost.

With a maintained dense inverse, initialization takes $\mathcal{O}(n^3)$ arithmetic
operations and each single flip costs $\mathcal{O}(n^2)$, including updating all
steepest scores via Subsection~\ref{sec:potential}. Thus the fixed-parameter
linear greedy bounds give $\mathcal{O}(n^3)$ arithmetic work. A fresh factorization
at every flip would instead give a coarser $\mathcal{O}(n^4)$ implementation.
Elimination-guided also admits $\mathcal{O}(n^3)$ dense arithmetic work, but has
different constants and auxiliary work. These statements do not address
floating-point stability of repeated rank-one updates. Input-bit dependence
must be included when counting bit operations. Computing $A^{-1}$ or
applying it is not free when the original input is $A$.

\subsection{Scope and open questions}
The results concern the specified selection rules and matrix classes.
The exponential least-index trajectory and the steepest six-cycle
illustrate different limitations, and each must be interpreted together
with the assumptions under which it occurs.

Open questions include the optimal parameter-dependent leading constants,
linear bounds under uniform polarity alone, and efficient certification
outside the present sufficient classes. The historical relationship of the
least-index construction to classical complementary-pivot lower bounds
requires a separate comparison and no claim of an exhaustive priority search
is made. The lower bound concerns a specified rule, not the difficulty of
solving the underlying small-norm AVE class.

\section{Numerical examples}\label{sec:numerics}
In this section, we validate the theoretical results through numerical examples. The following seven variants of the flip rule are compared: least-index (LI), cyclic scanning (CS),  Gauss-Southwell (GS), steepest (ST), elimination-guided selection (EG), randomized $1$-flip (RA) (the seed is $20260927$.) and full flipping (FN).  An FN update flips all currently mismatched signs and
performs the next candidate solve and initialization is excluded from its
update count, as it is for every other method. Thus an FN update can
change several signs, whereas one update of any other method changes
exactly one sign. No assertion of equal cost per update is
intended. All computations are done in MATLAB.

\subsection{Exact arithmetic on the exponential family}
We use \eqref{eq:ea}--\eqref{eq:construction} with $t=1/32$ and the natural
initialization. Candidate solutions and flip gains in this experiment are computed in exact rational arithmetic, using MATLAB symbolic arithmetic and arbitrary-precision integer recurrences for the bidiagonal systems. Sign mismatches are identified by the strict condition \(z_i x_i(z)<0\), without a numerical tolerance.
EG performs exact reduced-system elimination. All seven methods are actually
executed and the least-index counts are not substituted from the Theorem~\ref{thm:lower}.
Every final AVE residual is verified to be identically zero.

\begin{table}[ht]
\centering\small
\caption{Exact rational experiment. The first seven method columns count
updates and only FN updates may flip multiple signs (see the eighth column). The final column counts
all CS inspections, including its stopping sweep.}\label{tab:exact}
\begin{tabular}{@{}rrrrrrrrrr@{}}
\toprule
$n$ & LI & CS & GS & ST & EG & FN & RA & FN flips & CS checks\\
\midrule
4 & 7 & 5 & 1 & 1 & 1 & 3 & 5& 5 & 16  \\
6 & 31 & 11 & 1 & 1 & 1 & 5 &3& 11  &  36 \\
8 & 127 & 19 & 1 & 1 & 1 & 7&1 &  19 &64\\
10 & 511 & 29 & 1 & 1 & 1 & 9 &27  & 29 & 100 \\
12 & 2047 & 37 & 1 & 1 & 1 & 11& 7 & 37 & 144\\
14 & 8191 & 49 & 1 & 1 & 1 & 13 &1 & 49 & 196\\
\bottomrule
\end{tabular}
\end{table}

Table~\ref{tab:exact} reproduces the LI count $2^{n-1}-1$ exactly:
it increases from $7$ to $8191$ as $n$ grows from $4$ to $14$, whereas
GS, ST, and EG consistently require one flip. The same instances
therefore exhibit a large separation caused by coordinate selection.
CS takes between $5$ and $49$ flips and exactly $n^2$ inspections,
including its stopping sweep. It can reach a compatible signature
before completing those inspections. FN takes $n-1$ batch updates,
but its total number of individual sign changes equals the CS flip
count in every displayed row. Fewer updates thus need not mean fewer
sign changes. The RA counts vary substantially across dimensions.
Exact arithmetic is essential here because the deliberately tiny
right-hand-side components would make a fixed absolute mismatch
tolerance inappropriate.

\subsection{Paired random instances}
For each $n\in\{32,64,128\}$ and $\gamma\in\{0.2,0.45\}$, we generate
10 instances in each of two families. In the dense family the entries of
$G$ are independent standard normal variables. In the banded family we
set $G_{ij}=0$ when $|i-j|>2$. In both cases we put
\[
 A^{-1}=\frac{\gamma G}{\max\{\|G\|_1,\|G\|_\infty\}},
 \qquad d_i\ \hbox{independent uniform on }[-1,1].
\]
Thus all seven methods operate within their proven norm regimes.
We test both $z^{(0)} =\sgn(d)$ and an independent uniform random signature.
Every method receives the same matrix, right-hand side, and initial
signature for each paired instance. The seed is $20260927$.
There are $120$ systems, two initializations per system, and $1680$ solver
runs in total. These matrices are nonsingular with probability one, and
we work directly with the normalized equation without forming $A$.

Single-flip methods initialize an inverse and use rank-one updates for
both the inverse and candidate vector. ST obtains its scores from the
current inverse. EG additionally maintains its reduced system.
FN uses a fresh dense solve for each candidate. The code is a reference
implementation, not a tuned sparse or factorization-update implementation.
Timing includes initialization and all selection work but excludes instance
generation and final residual validation. Each solver is run once per
paired instance, with execution order rotated across trials. Medians
below therefore describe this implementation and environment rather than
an intrinsic ordering of the algorithms.

\begin{table}[ht]
\centering\small
\caption{Mean update counts, pooled over the $40$ systems per dimension
(two families, two norm levels, ten trials). N: natural initialization;
R: random signature. FN counts block updates.}\label{tab:randomupdates}
\begin{tabular}{@{}lrrrrrrrr@{}}
\toprule
Init. & $n$ & LI & CS & GS & ST & EG & FN & RA\\
\midrule
N & 32 & 0.53 & 0.53 & 0.53 & 0.53 & 0.53 & 0.35 & 0.53\\
N & 64 & 0.85 & 0.85 & 0.85 & 0.85 & 0.85 & 0.58 & 0.85\\
N & 128 & 1.58 & 1.58 & 1.58 & 1.58 & 1.58 & 0.65 & 1.58\\
R & 32 & 17.53 & 17.23 & 16.93 & 16.93 & 16.93 & 1.45 & 17.48\\
R & 64 & 34.23 & 33.13 & 32.43 & 32.43 & 32.43 & 1.65 & 33.33\\
R & 128 & 66.80 & 65.50 & 64.40 & 64.40 & 64.40 & 1.73 & 66.20\\
\bottomrule
%
\end{tabular}
\end{table}

\begin{table}[ht]
\centering\small
\caption{Mean elapsed time in milliseconds for the same pooled runs,
including initialization. }
\label{tab:times}
\begin{tabular}{@{}lrrrrrrrr@{}}
\toprule
Init. & $n$ & LI & CS & GS & ST & EG & FN & RA \\
\midrule
N & 32 & 0.13  &  0.13 &  0.12 &  0.13&  0.43&    0.20 & 0.15\\
N & 64 & 0.22&  0.23& 0.24&   0.23&    0.81& 0.28 & 0.22\\
N & 128 & 0.44&  0.49&   0.46&    0.48&     2.23&     0.47 & 0.47\\
R & 32 & 0.24&    0.22&     0.25  &  0.33&    0.52&     0.27& 0.25\\
R & 64 & 0.50&   0.45 &     0.51&     0.67&     1.02&    0.37& 0.44\\
R & 128 & 2.88&    2.73&    3.21&    3.48 &    4.97&     0.67& 3.00\\
\bottomrule
%
\end{tabular}
\end{table}

All runs terminate within the implementation cap, and their normalized
residuals are evaluated as
\[
{\rm RRES} =  \frac{\|x-A^{-1}|x|-d\|_\infty}
 {(1+\|A^{-1}\|_\infty)\|x\|_\infty+\|d\|_\infty}.
\]
The residual
check is performed separately after termination. The maximum RRES over all $1680$ runs is $7.58 \times10^{-16}$.
CS uses at most $n^2$ inspections in every run.

Table~\ref{tab:randomupdates} shows a strong initialization effect.
With natural initialization, all six single-flip rules have the same
reported mean counts, from $0.53$ to $1.58$, so these samples leave
little room for coordinate selection to affect the count. From random
signatures, GS, ST, and EG again have equal reported means, now close
to $n/2$: $16.93$, $32.43$, and $64.40$. This is consistent with
roughly half of the initial signs needing correction. LI, CS, and RA
use slightly more flips on these pooled samples. The exponential LI
behavior is therefore not typical of this sampled family, although
the experiment does not determine average-case complexity. FN takes
only $1.45$--$1.73$ batch updates from random signatures, with multiple
sign changes permitted per update.

Table~\ref{tab:times} shows why update counts alone do not predict
elapsed time. At $n=128$ with random initialization, CS takes more
flips than GS and ST but less time ($2.73$ ms versus $3.21$ and
$3.48$ ms), consistent with additional selection overhead in this
implementation. EG has the same mean flip count as GS and ST but
takes $4.97$ ms when explicit elimination is included. FN takes
$0.67$ ms in this case, whereas it is slower than GS for $n=32$
with natural initialization ($0.20$ ms versus $0.12$ ms).
When few updates are needed, initialization and selection overhead
can account for a larger share of the total cost. These observations
explain the measured tradeoffs between single flips and batch solves and
they do not imply a universal timing advantage for any method.

\section{Conclusion}\label{sec:conclusion}
Coordinate selection determines the worst-case behavior of
single-flip methods for AVEs. Rohn's least-index rule can require
$2^{n-1}-1$ flips from natural initialization even when both inverse
norms are arbitrarily small, whereas the proposed Gauss--Southwell
and steepest single-flip rules solve the same instances in one flip.
Under positive polarity and $r=\|A^{-1}\|_\infty<1$, both greedy rules
satisfy $K\le\lfloor s\kappa(1+r)/(1-r)\rfloor$. This yields a sharp
$\mathcal{O}(n)$ dimension dependence on the stated class with fixed
inverse-norm margins. Stronger sign-correction conditions give the
optimal count of one flip per initially incorrect nonzero solution
sign and the five-dimensional six-flip example shows the limits of
extending this guarantee.
Radons--Rump certificates also yield three polynomial guarantees under $\rho(|A^{-1}|)< 1/2$ (with positive diagonal scaling if needed): at most $n$ flips for elimination guidance,
$n^2$ inspections for cyclic scanning, and
$\mathcal{O}(n^2/m)$ expected batch updates for randomized $m$-flipping. Numerical examples illustrate these distinctions while separating iteration counts from implementation cost.

Several directions merit further investigation, including relaxing the current norm and spectral radius assumptions and developing other algorithms based on the potential optimization formulation.

\appendix
\section{A general polarity-only bound}\label{sec:upper}
Geometric contraction alone usually leads to a bound involving the
smallest positive gap. The exact identity~\eqref{eq:exactgap} has additional
structure: all positive gaps lie in a union of only $s$ logarithmic
intervals, where $s$ is the number of nonzero solution coordinates.
Their total length is controlled independently of the coordinate magnitudes.

\begin{lemma}[Logarithmic interval length]\label{lem:interval}
Let $0<a_1\le\cdots\le a_s$, $\kappa\ge1$, and
$B_j=\sum_{i=1}^j a_i$. The union
\[
 \bigcup_{j=1}^s[\log a_j,\log(\kappa B_j)]
\]
has at most $s$ connected components and total length at most
\begin{equation}\label{eq:Ls}
 L_s=\log\kappa+(s-1)\log(1+\kappa).
\end{equation}
\end{lemma}

\begin{proof}
The first interval has length $\log\kappa$.
When the $j$th interval is added, the preceding rightmost endpoint is
$\log(\kappa B_{j-1})$. If $a_j\le\kappa B_{j-1}$, it overlaps the
preceding interval and adds length
\[
 \log(B_j/B_{j-1})=\log(1+a_j/B_{j-1})\le\log(1+\kappa).
\]
Otherwise it is disjoint and its entire length is
\[
 \log(\kappa B_j/a_j)
 =\log(\kappa+\kappa B_{j-1}/a_j)<\log(1+\kappa).
\]
Summing these contributions proves the bound \eqref{eq:Ls}. The bound on the number of connected components follows because there are $s$ intervals.
\end{proof}

\begin{theorem}[Exact termination from contraction]\label{thm:termination}
Suppose that Assumption~\ref{assu:regular} and Assumption~\ref{ass:polarity} hold.
Let $s={\rm card}(\{i\in [n]: x_i^*\ne 0\})$. Suppose a mismatch-flip trajectory satisfies
\begin{equation}\label{eq:ca}
 G(z^{(\ell+1)})\le(1-\mu)G(z^{(\ell)})
\end{equation}
at each nonterminal state, for a fixed $0<\mu\le1$.
If $s=0$, it makes no updates. If $\mu=1$, it makes at most one update.
Otherwise its number $K$ of updates satisfies
\begin{equation}\label{eq:termination}
 K\le s+\left\lceil
 \frac{\log\kappa+(s-1)\log(1+\kappa)}{-\log(1-\mu)}
 \right\rceil
 \le s+\left\lceil\frac{L_s}{\mu}\right\rceil.
\end{equation}
\end{theorem}

\begin{proof}
If $s=0$, then $x^*=0$, hence $b=0$ and every candidate equals zero.
If $\mu=1$, the first nonterminal update gives zero gap.

Now suppose $s\ge1$ and $0<\mu<1$.
Sort the nonzero solution magnitudes as
\[
    0<a_{k_1}\le\cdots\le a_{k_s},
    \qquad
    B_k=\sum_{j=1}^k a_{k_j}.
\]
For any nonterminal sign state $z$, the set
\[
    D(z)=\{i\in [n]:z_i x_i^*<0\}
\]
is nonempty.
Let $k_{m}$ be the largest index, in the chosen ordering
of the nonzero solution components, belonging to
$D(z)$. By~\eqref{eq:exactgap},
\[
    2\alpha a_{k_{m}}
    \le G(z)
    \le2\beta B_m.
\]
Consequently,
\[
    \log\frac{G(z)}{2\alpha}
    \in
    \bigcup_{j=1}^s
    [\log a_{k_j},\log(\kappa B_j)].
\]

Let $K$ be the number of updates before termination.
The positive-gap states are $z^{(0)},\ldots,z^{(K-1)}$.
For consecutive positive-gap states, define
\[
    h_k:=\log\frac{G(z^{(k)})}{2\alpha}.
\]
Condition~\eqref{eq:ca} implies
\[
    h_k-h_{k+1}\ge\eta,
    \qquad
    \eta:=-\log(1-\mu)>0.
\]
Thus the $K$ values $h_k$ form a decreasing sequence
with separation at least $\eta$.

By Lemma~\ref{lem:interval}, these values
belong to a union of at most $s$ closed intervals
whose total length is at most
\[
    L_s:=\log\kappa+(s-1)\log(1+\kappa).
\]
A closed interval of length $\ell$ contains at most
$1+\ell/\eta$ elements of such a sequence.
Summing over the connected components gives
\[
    K \le s+\frac{L_s}{\eta}.
\]
Taking a ceiling and using
$-\log(1-\mu)\ge\mu$ proves~\eqref{eq:termination}.
\end{proof}

\begin{theorem}[Bounds for specified selection rules]\label{thm:upper}
Under Assumption~\ref{ass:polarity} and Assumption~\ref{ass:polarity}, from every initial
signature and for $s\ge1$, a $c(n)$-capture rule satisfies
\begin{equation*}
K \le s+\left\lceil\frac{\kappa L_s}{c(n)}
       \right\rceil
   = \mathcal{O}\!\left(\frac{s\kappa}{c(n)}\log(1+\kappa)\right).
\end{equation*}
Both Gauss-Southwell and steepest single-flip selection satisfy
\begin{equation*}
 K \le s+\left\lceil n\kappa
       L_s\right\rceil
   = \mathcal{O}(ns\kappa\log(1+\kappa)).
\end{equation*}
These estimates require no rationality assumption or lower bound on
nonzero solution magnitudes.
\end{theorem}

\begin{proof}
Apply Proposition~\ref{prop:contraction} and
Theorem~\ref{thm:termination} with $\mu=c(n)/\kappa$ or
$\mu=1/(n\kappa)$. The case $\mu=1$ is covered by its separate
one-update conclusion, which also satisfies the displayed bounds.
Use $L_s\le s\log(1+\kappa)$ for the asymptotic estimates.
\end{proof}

\section*{Funding}
Research of Chen C. R. was partly supported by the Natural Science Foundation of Fujian Province (Grant No. 2025J01673).  Research of Xia Y. was partly supported by the National Natural Science Foundation of China (Grant No. 12631012).

\section*{Declarations}

\subsection*{Conflict of interest}
The authors declare that they have no conflict of interest.

\subsection*{Data availability}
This paper does not analyze or generate any dataset.

\end{document}